\documentclass[reqno, 12pt]{amsart}
\usepackage[margin=1in]{geometry}
\usepackage{amsthm, amsmath, amsfonts, amssymb}
\usepackage{enumerate, mathabx}
\usepackage[hyphens]{url}
\usepackage{hyperref}
\usepackage[hyphenbreaks]{breakurl}

\usepackage{tikz}
\usepackage{subfig}
\usetikzlibrary{positioning,matrix}

\newcommand{\F}[0]{\mathbb{Q}_q}
\renewcommand{\O}[0]{\mathbb{Z}_q}

\newcommand{\Z}[0]{\mathbb{Z}}
\newcommand{\R}[0]{\mathbb{R}}
\newcommand{\Q}[0]{\mathbb{Q}}
\newcommand{\C}[0]{\mathbb{C}}
\newcommand{\N}[0]{\mathbb{N}}
\newcommand{\T}[0]{\mathbb{T}}
\newcommand{\Om}[0]{\Omega}
\newcommand{\ta}[0]{\theta}
\newcommand{\bs}[0]{\setminus}
\newcommand{\ld}[0]{\lambda}
\newcommand{\vep}[0]{\varepsilon}

\newcommand{\lsm}[0]{\lesssim}

\newcommand{\wh}[1]{\widehat{#1}}

\newcommand{\mc}[1]{\mathcal{#1}}
\newcommand{\ov}[1]{\overline{#1}}
\newcommand{\wt}[1]{\widetilde{#1}}
\newcommand{\st}[1]{\substack{#1}}

\newcommand{\mf}[1]{\mathfrak{#1}}

\newcommand{\nms}[1]{\| #1 \|}

\newcommand{\Mod}[1]{\ (\mathrm{mod}\ #1)}

\def\Xint#1{\mathchoice
{\XXint\displaystyle\textstyle{#1}}%
{\XXint\textstyle\scriptstyle{#1}}%
{\XXint\scriptstyle\scriptscriptstyle{#1}}%
{\XXint\scriptscriptstyle\scriptscriptstyle{#1}}%
\!\int}
\def\XXint#1#2#3{{\setbox0=\hbox{$#1{#2#3}{\int}$ }
\vcenter{\hbox{$#2#3$ }}\kern-.6\wd0}}
\def\dashint{\Xint-}

\usetikzlibrary{backgrounds}
\tikzset {
  highlight/.style = {gray, opacity=0.75},
  digit/.style = { minimum height=5mm, minimum width=5mm, anchor=center },
}

\newcounter{row}
\newcounter{col}

\newcommand\setrow[9]{
    \setcounter{col}{1}
    \foreach \n in {#1, #2, #3, #4, #5, #6, #7, #8, #9} {
        \edef\x{\value{col} - 0.5}
        \edef\y{9.5 - \value{row}}
        \node[digit,name={\arabic{row}-\arabic{col}}] at (\x, \y) {\n};
        \stepcounter{col}
    }
    \stepcounter{row}
}

\def\highlightrectangle#1#2#3#4{
\begin{pgfonlayer}{background}
\fill[highlight] (#1-#2.north west) rectangle (#3-#4.south east);
\end{pgfonlayer}
}

\newtheorem{thm}{Theorem}[section]
\newtheorem{lemma}[thm]{Lemma}
\newtheorem{prop}[thm]{Proposition}
\newtheorem{cor}[thm]{Corollary}

\theoremstyle{remark}
\newtheorem{rem}{Remark}

\usepackage{color}

\author{Zane Kun Li}
\address{Department of Mathematics, University of Wisconsin-Madison, Madison, WI 53706, USA}
\email{zkli@wisc.edu}

\title{An introduction to decoupling and harmonic analysis over $\Q_p$}

\begin{document}
\begin{abstract}
The goal of this expository paper is to provide an introduction to decoupling by working
in the simpler setting of decoupling for the parabola over $\Q_p$.
Over $\Q_p$, commonly used heuristics in decoupling
are significantly easier to make rigorous over $\Q_p$ than over $\R$ and such decoupling theorems over $\Q_p$ 
are still strong enough to derive interesting number theoretic conclusions.
\end{abstract}

\maketitle

\section{Introduction}
\subsection{Motivation}
For an interval $J \subset [0, 1]$, let
\begin{align}\label{tajdef}
\ta_{J} := \{(\xi, \eta) \in \R^2: \xi \in J, |\eta - \xi^2| \leq |J|^2\}
\end{align}
and let $f_{J}$ be defined such that $\wh{f_J} := \wh{f}1_{J \times \R}$.
For $\delta \in \N^{-1}$, let $P_{\delta}$ be a partition of $[0, 1]$ into
intervals of length $\delta$.
For $2 \leq p \leq \infty$, let $D_{p, \R}(\delta)$ be the best constant such that
\begin{align}\label{decdef}
\nms{f}_{L^{p}(\R^2)} \leq D_{p, \R}(\delta)(\sum_{K \in P_{\delta}}\nms{f_K}_{L^{p}(\R^2)}^{2})^{1/2}
\end{align}
for all $f$ with Fourier transform supported in $\bigcup_{K \in P_{\delta}}\ta_K$.
A special case of Bourgain and Demeter's paraboloid decoupling theorem \cite{BD15, BD17}
gives that 
$$D_{p, \R}(\delta) \lsm_{p, \vep} \delta^{-\vep}(1 + \delta^{-(\frac{1}{2} - \frac{3}{p})}).$$
By interpolation, it suffices to prove the critical estimate that $D_{6, \R}(\delta) \lsm_{\vep} \delta^{-\vep}$.
In 2020, Guth, Maldague, and Wang in \cite{GMW} 
improved this estimate to $\lsm (\log \delta^{-1})^{C}$
for an unspecified large constant $C$.
One application of their estimate is the following:
Let $K(N)$ be the best constant such that 
\begin{align}\label{kndef}
\nms{\sum_{n = 1}^{N}a_n e(nx + n^2 t)}_{L^{6}_{x, t}(\T^2)} \leq K(N)(\sum_{n = 1}^{N}|a_n|^{2})^{1/2}
\end{align}
for all $\{a_n\}_{n = 1}^{N} \subset \C$ and where here $e(z) := e^{2\pi i z}$.
They were able to show that $K(N) \lsm (\log N)^{C}$ for the same unspecified large constant $C$.
Before this, the best known upper bound was due to Bourgain in \cite{BourgainGAFA93}
who showed an upper bound of $\exp(O(\frac{\log N}{\log\log N}))$ using the divisor bound and a lower bound of $(\log N)^{1/6}$ by analyzing major arcs.

Recently, instead of working over $\R$, decoupling has been studied over $\F$ (we use $\F$ instead of $\Q_p$ as we 
reserve $p$ for the $L^p$ norm), see for example \cite{CHLMRY} and \cite{GLY21}. The prime $q$ throughout will be an odd prime.
In \cite{GLY21}, with Shaoming Guo and Po-Lam Yung, we optimized Guth, Maldague, and Wang's argument and worked over $\F$ instead of $\R$
to prove a parabola decoupling theorem over $\F$ which then was able to reduce $C$ down to $2+$, that is, $K(N) \lsm_{\vep} (\log N)^{2 + \vep}$.

Similar to the advantages of working with the Walsh model in time frequency analysis (see for example \cite{thiele}), there are some clear advantages in working over $\F$ instead of $\R$. The following are some heuristics
that are used when thinking about decoupling over $\R$ but need to be made rigorous when writing out the proofs. 
This often involves working with smooth cutoffs or convolutions to make all the ``essentially" statements rigorously true and can be somewhat
technical.
Over $\F$, however, the statements below are rigorously true with little modification.

\begin{enumerate}[$(i)$]
\item The neighborhood of the parabola $\ta_{J}$ is essentially a parallelogram.

\item Balls and squares are essentially the same and so the square $[0, R]^2$ is essentially the same as the square $[-R, R]^2$ which in turn
is essentially the same as the ball $B(0, R)$. Furthermore, $O(1)$ dilates of geometric objects such as a balls and rectangles are also essentially the same,
and so we don't distinguish between $B(0, R)$ and $B(0, 2R)$.

\item Consider an $R \times R^2$ rectangle centered at the origin pointing in the direction of (say) $(-1, 1)$.
The square $[0, R^2]^2$ can essentially be partitioned by $R$ many translates of these rectangles.

\item The Fourier transform of $1_{B(0, 1)}$ is essentially $1_{B(0, 1)}$.

\item (Small angle approximation) If $|\xi| \ll 1$, then $e(\xi) \approx 1$.

\item (Local constancy property) If $f$ is Fourier supported in $\ta$, then $|f|$ is essentially constant on translates of the polar (or dual) set $\ta^{\ast} := \{x \in \R^2: |x \cdot \xi| \leq 1 \text{ for all } \xi \in \ta\}$. Thus if $J = [a, a + \delta]$ and $\ta = \ta_{J}$, then $|f|$ is essentially constant on translates of the ``dual tube" $\{(x, y) \in \R^2: |x + 2ay| \leq \delta^{-1}, |y| \leq \delta^{-2}\}$. The $a$ in $x + 2ay$ can be replaced with any point in $J$.
\end{enumerate}
Additionally, decoupling theorems over $\F$ are still strong enough to imply
the applications of decoupling to exponential sums and solution counting such as in Vinogradov's Mean Value Theorem
and \eqref{kndef} above as often only even $L^p$ norms are considered. Thus the author believes decoupling over $\F$ provides an ideal setting for someone just starting out in the area.

\subsection{Decoupling for the parabola over $\F$}
We now redefine our definitions in \eqref{tajdef} and \eqref{decdef} to work over $\F$.
For $J \subset \O$, let
\begin{align}\label{tajqp} 
\ta_{J} := \{(\xi, \eta) \in \F^2: \xi \in J, |\eta - \xi^2| \leq |J|^2\}
\end{align}
and let $f_J$ be defined such that $\wh{f_J} := \wh{f}1_{J \times \F}$.  
For $\delta \in q^{-\N}$, let $P_{\delta}$ be a partition of $\O$ into intervals 
of length $\delta$.
Finally, we let $D_{p}(\delta) := D_{p, \F}(\delta)$ be defined such that
\begin{align*}
\nms{f}_{L^{p}(\F^2)} \leq D_{p, \F}(\delta)(\sum_{K \in P_{\delta}}\nms{f_K}_{L^{p}(\F^2)}^{2})^{1/2}
\end{align*}
for all $f$ with Fourier transform supported in $\bigcup_{K \in P_{\delta}}\ta_K$.
We will prove the following $\F$ analogue of Bourgain and Demeter's decoupling theorem
applied to the parabola.
\begin{thm}\label{main}
For $\delta \in q^{-\N}$, we have
$$D_{p, \F}(\delta) \lsm_{p, q, \vep} \delta^{-\vep}(1 + \delta^{-(\frac{1}{2} - \frac{3}{p})}).$$
\end{thm}
Once again by interpolation, it suffices to prove that $D_{6, \F}(\delta) \lsm_{q, \vep} \delta^{-\vep}$.
The proof we follow is the efficient congruencing inspired proof of parabola
decoupling over $\R$ due to myself in \cite{Li18}.
Following the proofs of the Main Conjecture in Vinogradov's Mean Value Theorem
using the harmonic analysis method of decoupling by Bourgain, Demeter, and Guth
in \cite{BDG} and the number theoretic method of efficient congruencing by Wooley
\cite{WooleyNested}, it is an interesting question to determine how these
two methods are related. This study led to new proofs of decoupling for the parabola \cite{Li18},
cubic moment curve \cite{GLY19}, and the degree $k$ moment curve \cite{GLYZK}, corresponding to
\cite[Section 4.3]{Pierce}, \cite{HB-Cubic}, and \cite{WooleyNested}, respectively.
See also \cite{CDGJLM} for proof of decoupling for fractal sets on the parabola which was
inspired from efficient congruencing
for ellipsephic sets \cite{Biggs19}. Additionally see \cite{CHLMRY} for an interpretation
into decoupling of a classical argument (see for example \cite[Chapter 5]{V} or \cite[Theorem 13]{Tao254A}) that gave partial progress towards the Main Conjecture
in Vinogradov's Mean Value Theorem.

This paper is organized as follows: In Sections \ref{qpfacts} and \ref{moregeom}, we review the needed facts for $\F$
that are relevant to our paper. In particular, in Section \ref{moregeom}, we state the version of the locally constant property/wavepacket decomposition
that we will use.
In Section \ref{basics}, we show that Theorem \ref{main} is sharp up to the $\vep$ in the exponent, prove an interpolation theorem that allows use to reduce Theorem \ref{main}
to just showing the $p = 6$ case, and show that Theorem \ref{main} is strong enough to imply an estimate for $K(N)$ in \eqref{kndef}.
In Section \ref{l2l6pf}, we prove $D_{6, \F}(\delta) \lsm_{q, \vep} \delta^{-\vep}$ save for one key step (Proposition \ref{keystep1}). Finally, in Section \ref{fiveproofs},
we give five different proofs of Proposition \ref{keystep1}.

\subsection*{Notation}
We say that
$f$ is Fourier supported in $\Om$ if its Fourier transform $\wh{f}$ is supported
in $\Om$. Given an interval $I$, we let $P_{\delta}(I)$ be the partition of $I$ into intervals of length $\delta$.
Given two positive expressions $X$ and $Y$, we write
$X \lsm Y$ if $X \leq CY$ for some absolute constant $C$. If $C$ depends on some additional parameter $A$,
then we write $X \lsm_{A} Y$. We write $X \sim Y$ if $X \lsm Y$ and $Y \lsm X$.
By writing $f(x) = O(g(x))$, we mean $|f(x)| \lsm g(x)$.

\subsection*{Acknowledgements} The author is supported
by NSF grants DMS-2037851 and DMS-2311174. This expository paper evolved from the slides of a talk the author gave on March 26, 2022 at the AMS Spring Central Sectional Meeting 
and also the notes for a series of lectures given by the author from April to May 2022 at the National Center for Theoretical
Science (NCTS) in Taipei, Taiwan. A recording of this lecture series can be viewed at \url{https://youtu.be/_ML8jAu5cqo&list=PLQZfZKhc0kiAXeVcTAnYT9_YwpQ13SMB5}.
The author would like to thank NCTS for the kind hospitality during his visit.
The author would also like to thank the anonymous referee for their comments and suggestions in improving the exposition.

\section{Some basic facts about $\F$}\label{qpfacts}
We will summarize all the relevant facts about $\F$ in this section and the next.
See the Introduction and Section 2 of \cite{CHLMRY}, Section 2 of \cite{GLY21}, Chapters 1 and 2 of \cite{Taibleson}, and
Chapter 1, Sections 1 and 4 of \cite{VVZ} for additional discussion
of analysis on $\F$. The exposition in this section and the next borrows heavily from the presentation in \cite{CHLMRY} and \cite{GLY21}.

\subsection{Definitions and basic geometric observations}
The field $\F$ is the completion of $\Q$ under the $q$-adic norm,
defined by $|0| = 0$ and $|q^{a}b/c| = q^{-a}$ if $a \in \Z$, $b, c \in \Z\bs\{0\}$
and $q$ is relatively prime to both $b$ and $c$. Then $\F$ can be identified
with the set of all formal series 
\begin{align*}
\F = \{\sum_{j = k}^{\infty}a_{j}q^{j}: k \in \Z, a_{j} \in \{0, 1, \ldots, q - 1\} \text{ for every } j \geq k\}
\end{align*}
and the $q$-adic norm on $\F$ satisfies $|\sum_{j = k}^{\infty}a_{j}q^{j}| = q^{-k}$
if $a_{k} \neq 0$.
A nice article illustrating one way to visualize $\Q_3$ can be found
in \cite{quanta}.
Strictly speaking we should be writing $|\cdot|_{q}$ instead of $|\cdot|$,
but we omit this dependence since we will only be dealing with one fixed $q$.
The $q$-adic norm on $\F$ induces a norm on $\F^2$, which we
denote also by $|\cdot|$ by abuse of notation, via $|(x, y)| := \max(|x|, |y|)$.
The key property obeyed by the $q$-adic norm is the ultrametric triangle inequality:
$|x + y| \leq \max(|x|, |y|)$ with equality if $|x| \neq |y|$. 

We adopt all our geometry terminology in analogy from the real setting.
An interval $I(a, r)$ in $\F$ is a set of the form $\{\xi \in \F: |\xi - a| \leq r\}$ where
$a \in \F$ and $r \geq 0$; $r$ will then be called the length of the interval. 
Given an interval
$I$, we will use $|I|$ to denote the length of $I$. 
A square in $\F^2$ of side length $r$ is then the product of 2 intervals
in $\F$ of length $r$. Finally, we define the ball $B(a, r) := \{\xi \in \F^2: |\xi - a| \leq r\}$ where $a \in \F^2$ and $r \geq 0$
and will say this ball has radius $r$, in analogy to the real setting.

The ring of integers
$\O$ coincides with the unit interval $I(0, 1) = \{\xi \in \F: |\xi| \leq 1\}$.
Elements in $\O$ are exactly the elements in $\F$
of the form $\sum_{j \geq 0}a_{j}q^{j}$
where $a_{j} \in \{0, 1, \ldots, q - 1\}$.
As each positive integer has a base $q$ representation, we may
embed $\N$ into $\O$. Identifying $-1$ with the formal power
series $\sum_{j \geq 0}(q - 1)q^{j}$ in $\O$ allows us to embed
$\Z$ into $\O$. Inverting formal power series and that the 
addition of two formal power series is still a formal power series
shows that $\Q$ embeds into $\F$.

Having given the definitions of $\F$ and $\O$, we now give some basic geometric observations.
Intervals of length $\delta$ partition $[0, 1]$ into $\delta^{-1}$ many intervals.
Not only is the same is true in $\O$, but also two distinct intervals are always separated by a positive distance.
\begin{lemma}\label{lengthpart}
Fix $k \in \N$. For $a \in \{0, 1, \ldots, q^{k} - 1\}$, the 
intervals $I(a, q^{-k}) = \{\xi \in \F: |\xi - a| \leq q^{-k}\}$ partition $\O$
into $q^{k}$ many disjoint intervals and each pair of intervals are separated
by a distance at least $q^{-k + 1}$.
\end{lemma}
\begin{proof}
The $\xi \in \F$ such that $|\xi - a| \leq q^{-k}$
are exactly the $\xi \in \O$ such that $\xi \equiv a \Mod{q^k}$ when we write $\xi$ as a formal power series.
Since $a$ ranges from $0$ to $q^{k} - 1$, the $I(a, q^{-k})$
partition $\O$ into $q^{k}$ many disjoint intervals of length $q^{-k}$.

Next, suppose $|\xi_{1} - a| \leq q^{-k}$ and $|\xi_{2} - b| \leq q^{-k}$
for some $a \neq b$. Since $|a - b| \geq q^{-k + 1}$ and
$|(\xi_{1} - \xi_{2}) - (a - b)| \leq q^{-k}$, it follows
from the equality case of the ultrametric triangle inequality
that $|\xi_{1} - \xi_{2}| = |a - b| \geq q^{-k + 1}$.
\end{proof}

Next, from how $|\cdot|$ is defined on $\F^2$, squares of side length $r$ and balls of radius $r$ are the same. More precisely:
\begin{lemma}
Fix an arbitrary $(A, B) \in \F^2$. The square
$\{(x, y) \in \F^2: |x - A| \leq r, |y - B| \leq r\}$
of side length $r$ is the same as the ball
$\{(x, y) \in \F^2: |(x, y) - (A, B)| \leq r\}$ of radius $r$.
\end{lemma}

By the ultrametric triangle inequality, intervals can be centered at any point in the interval.
\begin{lemma}\label{center}
Fix an interval $I(x_0, r)$. For any $b \in I(x_0, r)$, we have
$I(x_0, r) = I(b, r)$.
\end{lemma}
\begin{proof}
Fix arbitrary $b \in I(x_0, r)$, then $|b - x_0| \leq r$. 
We only show that $\{\xi \in \F: |\xi - b| \leq r\} \subset I(x_0, r)$.
But this follows from the ultrametric triangle inequality:
$|\xi - x_0| = |\xi - b + b - x_0| \leq \max(|\xi - b|, |b - x_0|) \leq r.$
The reverse containment is true by interchanging the roles of $x_0$ and $b$.
\end{proof}

Lemma \ref{center} then is able to show that two intervals of the same length 
are either exactly the same or completely disjoint.
\begin{cor}\label{intintersect}
If two intervals (potentially of different lengths) intersect, then one is contained
inside the other.
\end{cor}
\begin{proof}
Let $I_1$ and $I_2$ be two intervals with $|I_1| = r \leq s = |I_2|$. By hypothesis,
we can find an $a \in I_1 \cap I_2$. By Lemma \ref{center},
$I_1 = I(a, r)$ and $I_2 = I(a, s)$
and hence $I_1 \subset I_2$.
\end{proof}

Corollary \ref{intintersect}  shows that intervals in $\F$ act like dyadic intervals in $\R$. 
Furthermore, Lemma \ref{lengthpart} shows that two distinct intervals of two (potentially different) lengths
are going to be at least separated by $q$ times the length of the smaller interval. 
We can also extend Corollary \ref{intintersect} to products of intervals.

\begin{cor}
If two squares (potentially of different side lengths) intersect, then one
is contained inside the other.
\end{cor}
\begin{proof}
Let $B = I \times J$ and $B' = I' \times J'$ with $|I| = |J| = r$ and $|I'| = |J'| = s$.
Without loss of generality we may assume that $r \leq s$. 
We can find a $(x_0, y_0) \in B \cap B'$. Therefore $x_0 \in I \cap I'$
and $y_0 \in J \cap J'$ and by Corollary \ref{intintersect}, $I \subset I'$ and $J \subset J'$.
Therefore $B \subset B'$.
\end{proof}

We end with the observation that sum of two balls of radius $r$ is still a ball of radius $r$ and
that the shear of a ball is still a ball. This should be compared to the real setting where
the sum of two balls of radius $r$ is a ball of radius $2r$ and the shear of a ball or a square is an ellipse or parallelogram, respectively.

\begin{lemma}\label{ballprop}
We have the following properties:
\begin{enumerate}[(i)]
\item For $a, b \in \F$, if $I =I(a, r)$
and $I' = I(b, r)$,
then $I + I' = I(a + b, r)$.
Similarly, for $c, d \in \F^2$, $B(c, r) + B(d, r) = B(c + d, r)$.
\item If $S_{a} = (\begin{smallmatrix} 1 & a \\ 0 & 1\end{smallmatrix})$
with $|a| \leq 1$, then $S_{a}B(0, 1) = B(0, 1)$. 
\end{enumerate}
\end{lemma}
\begin{proof}
To prove
the first part of $(i)$, if $\alpha$ and $\beta$ were such that $|\alpha - a| \leq r$
and $|\beta - b| \leq r$, then $|(\alpha + \beta) - (a + b)| \leq \max(|\alpha - a|, |\beta - b|) \leq r$ and hence $\alpha + \beta \in I(a + b, r)$.
On the other hand, if $\xi$ is such that $|\xi - (a + b)| \leq r$, then
write $\xi = (\xi - b) + b$. We have $\xi - b \in I$ and $b \in I'$.
The second part of $(i)$ follows from the first part of $(i)$
and the observation that balls of radius $r$ are just products of two intervals
of length $r$.

Next to prove $(ii)$,
since $S_{a}^{-1} = (\begin{smallmatrix} 1 & -a \\ 0 & 1\end{smallmatrix})$, it just suffices
to show that $S_{a}B(0, 1) \subset B(0, 1)$. 
But this follows from that $|x + 2ay| \leq \max(|x|, |2ay|) \leq 1$
for $|x|, |y| \leq 1$ where here we have also used that $q$ is an odd prime and so $|2| = 1$.
\end{proof}

\subsection{The Fourier transform on $\F$}
We now summarize the relevant facts about the Fourier transform over $\F$.
Let $\chi$ be the additive character of $\F$ that is equal to 1 on $\O$
and nontrivial on $q^{-1}\O$. This is the analogue of $e^{2\pi i x}$.
Up to isomorphism, the character is essentially unique,
given by $\chi(x) := e(\sum_{j = k}^{-1}a_{j}q^{j})$ if $x = \sum_{j \geq k}a_{j}q^{j}$. 
Therefore the small angle approximation in $\F$: 
``If $|\xi| \leq 1$, then $\chi(\xi) = 1$."
is rigorously true.

In $\F^d$, the Schwartz functions are finite linear combinations of characteristic
functions of cubes in $\F^d$. The Fourier transform of such a function $f$
is defined by
$$\wh{f}(\xi) := \int_{\F^d}f(x)\chi(-x \cdot \xi)\, dx$$
where here $dx$ is the Haar measure on $\F^d$ normalized so that $\O^d$ has volume 1.
The theory of integration and the Fourier transform in $\F$ is essentially the same
as in $\R$.
We refer the interested reader to \cite[Section 4]{VVZ} for more details.

The key property obeyed by the Fourier transform is that the
Fourier transform of the unit interval is the unit interval and so the indicator function of the unit
intervals plays the same role as Gaussians in $\R$. This should
be compared to the real setting where the Fourier transform of the unit interval is a sinc function. More precisely:

\begin{prop}\label{ballft}
We have that
$\wh{1}_{\O} = 1_{\O}$. 
\end{prop}
\begin{proof}
We need to show that
\begin{align}\label{fttarget}
\int_{\O}\chi(-x\cdot \xi)\, dx = 1_{\O}(\xi).
\end{align}
If $\xi \in \O$, then $|-x\xi| \leq 1$ and hence $-x \xi \in \O$.
Since $\chi$ is trivial on $\O$ and $\O$ has volume 1, it follows
that the left hand side of \eqref{fttarget} is equal to 1.

On the other hand if $\xi \not\in \O$, choose an $x_{0}$ such that
$|x_{0}| = 1$. Since if $|\xi| > 1$, then $|\xi| \geq q$, this implies that $|-x_{0}\xi| \geq q$ and so $\chi(-x_{0} \xi) \neq 1$.
Additionally, the ultrametric triangle inequality gives that $\O = x_0 + \O$. 
We then have
\begin{align*}
\int_{\O}\chi(-x \xi)\, dx &= \chi(-x_{0} \xi)\int_{\O}\chi(-(x - x_{0}) \xi)\, dx\\
& = \chi(-x_{0} \xi)\int_{x_0 + \O}\chi(-y \xi)\, dy = \chi(-x_{0} \xi)\int_{\O}\chi(-y  \xi)\, dy.
\end{align*}
But since $\chi(-x_{0} \xi) \neq 1$, we must have $\int_{\O}\chi(-x \xi)\, dx = 0$.
\end{proof}

Since $B(0, 1) = \O^2$, Proposition \ref{ballft} immediately shows that the Fourier transform of $1_{B(0, 1)}$ is also $1_{B(0, 1)}$.
Additionally, by a change of variables, this shows for example that the Fourier transform of $1_{B(0, R)}$ is $R^{2}1_{B(0, 1/R)}$.

\section{Geometry for parabola decoupling and wavepacket decomposition}\label{moregeom}
Having set up the basic geometric and analytic facts over $\F$, we now move to observations that are most relevant
to parabola decoupling over $\F$.

\subsection{The neighborhood of the parabola}\label{nghbd}
We first finish some easy consequences about $\ta_J$.
In addition to \eqref{tajqp}, for $J \subset \O$, let $\tau_{J} := B((a, a^2), |J|)$ for any $a \in J$.
The definition of $\tau_{J}$ is independent of the choice of $a \in J$ by
the ultrametric triangle inequality. Additionally,
$\ta_J \subset \tau_{J}$ since if $(\xi, \eta) \in \ta_{J}$,
then $|\xi - a| \leq |J|$ since $\xi \in J$ and
\begin{align*}
|\eta - a^2| = |\eta - \xi^2 + \xi^2 - a^2| \leq \max(|\eta - \xi^2|, |\xi - a||\xi + a|) \leq |J|
\end{align*}
where in the last inequality we have also used that $\xi, a \in \O$.

While $\ta_{J}$ is the $|J|^{2}$ neighborhood of the piece of parabola
above $J$, it also turns out at this scale, the curvature of $\ta_{J}$ is no longer visible.

\begin{lemma}\label{curvestraight}
For any $J \subset \O$, $\ta_{J}$ coincides with the parallelogram
\begin{align*}
\{(\xi, \eta) \in \F^2: |\xi - a| \leq |J|, |\eta - 2a\xi + a^2| \leq |J|^{2}\}
\end{align*}
where $a$ is any point in $J$.
\end{lemma}
\begin{proof}
Fix arbitrary $a \in J$. By Lemma \ref{center},
$J = \{\xi \in \F: |\xi - a| \leq |J|\}$. Next,
$|\eta - \xi^2| = |(\eta - 2a\xi + a^2) - (\xi - a)^{2}|$
and hence if $|\xi - a| \leq |J|$, then
$|\eta - \xi^2| \leq |J|^2$ if and only if
$|\eta - 2a\xi + a^2| \leq |J|^2$. 
\end{proof}

Finally, the following lemma will be useful later in computing the support of a convolution.
\begin{lemma}\label{supportsum}
For any $J \subset \O$, $\ta_{J} + B(0, |J|^{2}) = \ta_{J}$.
\end{lemma}
\begin{proof}
The $\supset$ direction is immediate. Let $(\xi_{0}, \eta_{0}) \in \ta_{J}$
and $(r, s) \in B(0, |J|^2)$. Write $J = \{\xi \in \F: |\xi - a| \leq |J|\}$
for some $a \in J$. Then $|(\xi_0 + r) - a| \leq \max(|\xi_0 - a|, |r|) \leq |J|$
and hence $\xi_0 + r \in J$. Next,
\begin{align*}
|(\eta_0 + s) - (\xi_{0} + r)^{2}| &= |(\eta_{0} - \xi_{0}^{2}) + (s - 2r\xi_{0} - r^2)|\\
& \leq \max(|\eta_{0} - \xi_{0}^{2}|, |s|, |2r\xi_{0}|, |r^2|) \leq |J|^2
\end{align*}
where here we have also used that $q$ is an odd prime and so $|2| = 1$.
\end{proof}

\subsection{Geometry of tubes and wavepacket decomposition}
Fix a $\delta \in q^{-\N}$. We define the shorthand 
$$\ta_{\delta} := \delta \O \times \delta^{2}\O \quad \text{and} \quad T_{\delta} := \delta^{-1}\O \times \delta^{-2}\O.$$
They are dual to each other in the sense that
$T_{\delta} = \{x \in \F^{2}: |x \cdot \xi| \leq 1 \text{ for all } \xi \in \ta_{\delta}\}.$
Since $\delta\O$ can be partitioned into $\delta^{-1}$ many intervals
of length $\delta^{2}$, it follows that $\ta_{\delta}$ is the union of
$\delta^{-1}$ many squares of side length $\delta^{2}$.
Similarly, $T_{\delta}$ is the union of $\delta^{-1}$ many squares of side
length $\delta^{-1}$.

For $a \in \O$, let $M_{a} = (\begin{smallmatrix} 1 & 0 \\2a & 1\end{smallmatrix})$.
Then by Lemma \ref{curvestraight}, for each $K \in P_{\delta}$, we have
\begin{align}\label{takma}
\ta_{K} = \begin{pmatrix} a \\a^2\end{pmatrix} + M_{a}\ta_{\delta}
\end{align}
for any $a \in K$.
The right hand side is independent of $a \in K$ since if $b \in K$, then
\begin{align*}
\begin{pmatrix} b \\ b^{2}\end{pmatrix} = \begin{pmatrix} a\\a^{2} \end{pmatrix} + \begin{pmatrix} 1 \\2a\end{pmatrix}(b - a) + \begin{pmatrix} 0 \\ 1\end{pmatrix} (b - a)^{2} \in \begin{pmatrix} a \\a^{2}\end{pmatrix} + M_{a}\ta_{\delta}
\end{align*}
and $M_{a} = M_{b}(\begin{smallmatrix} 1 & 0 \\2(a - b) & 1\end{smallmatrix})$
and this latter matrix preserves $\ta_{\delta}$ by the ultrametric triangle inequality.

For $K \in P_{\delta}$ and $a \in K$, let $T_{0, K}$ be the dual parallelogram
to $\ta_{K}$ centered at the origin given by
$$T_{0, K} = \{x \in \F^2: |x \cdot (\xi - \begin{pmatrix} a \\a^2\end{pmatrix})| \leq 1 \text{ for all } \xi \in \ta_{K}\}.$$
Using \eqref{takma},
\begin{align*}
T_{0, K} &= \{x \in \F^2: |x_{1} + 2ax_{2}| \leq \delta^{-1}, |x_{2}| \leq \delta^{-2}\} = \{x \in \F^2: M_{a}^{T}x \in T_{\delta}\} = M_{a}^{-T}T_{\delta}
\end{align*}
for any $a \in K$. This parallelogram depends only on $K$ but not on the choice
of $a \in K$ since for any other $b \in K$, 
$M_{a}^{-T} = M_{b}^{-T}(\begin{smallmatrix}1 & 2(b - a )\\0 & 1 \end{smallmatrix})$
and $|2(b - a)| \leq \delta$ and so this second matrix preserves $T_{\delta}$
by the ultrametric triangle inequality.
In analogy to the real setting, we will call $(-2a, 1)$ the direction of $T_{0, K}$.

Since $T_{\delta}$ is the union of $\delta^{-1}$ many squares of side length
$\delta^{-1}$, by Lemma \ref{ballprop}(ii), it follows that
$T_{0, K}$ is also the union of $\delta^{-1}$ many squares of side length $\delta^{-1}$.
Heuristically, it is useful to think of $T_{0, K}$ as a $\delta^{-1} \times \delta^{-2}$ ``tube" 
pointing in the direction $(-2a, 1)$.
\begin{center}
\begin{figure}[t!]
\captionsetup[subfigure]{labelformat=empty}
\centering
\subfloat[$T_{0, I(0, 1/3)}$]{
\begin{tikzpicture}[scale=0.5]
    \draw (0, 0) grid (9, 9);
    \draw[very thick, scale=3] (0, 0) grid (3, 3);

    \setcounter{row}{1}
    \setrow{}{}{}{}{}{}{}{}{}
    \setrow{}{}{}{}{}{}{}{}{}
    \setrow{}{}{}{}{}{}{}{}{}
	\setrow{}{}{}{}{}{}{}{}{}
	\setrow{}{}{}{}{}{}{}{}{}
	\setrow{}{}{}{}{}{}{}{}{}
	\setrow{}{}{}{}{}{}{}{}{}
	\setrow{}{}{}{}{}{}{}{}{}
	\setrow{}{}{}{}{}{}{}{}{}

	\node[anchor=center] at (-0.5, 0.5) {\tiny 00};
	\node[anchor=center] at (-0.5, 1.5) {\tiny 01};
	\node[anchor=center] at (-0.5, 2.5) {\tiny 02};
	\node[anchor=center] at (-0.5, 3.5) {\tiny 10};
	\node[anchor=center] at (-0.5, 4.5) {\tiny 11};
	\node[anchor=center] at (-0.5, 5.5) {\tiny 12};
	\node[anchor=center] at (-0.5, 6.5) {\tiny 20};
	\node[anchor=center] at (-0.5, 7.5) {\tiny 21};
	\node[anchor=center] at (-0.5, 8.5) {\tiny 22};

	\node[anchor=center] at (0.5, -0.5) {\tiny 00};
	\node[anchor=center] at (1.5, -0.5) {\tiny 01};
	\node[anchor=center] at (2.5, -0.5) {\tiny 02};
	\node[anchor=center] at (3.5, -0.5) {\tiny 10};
	\node[anchor=center] at (4.5, -0.5) {\tiny 11};
	\node[anchor=center] at (5.5, -0.5) {\tiny 12};
	\node[anchor=center] at (6.5, -0.5) {\tiny 20};
	\node[anchor=center] at (7.5, -0.5) {\tiny 21};
	\node[anchor=center] at (8.5, -0.5) {\tiny 22};

	\highlightrectangle{1}{1}{3}{3}
	\highlightrectangle{4}{1}{6}{3}
	\highlightrectangle{7}{1}{9}{3}
\end{tikzpicture}
}
\subfloat[$T_{0, I(1, 1/3)}$]{
\begin{tikzpicture}[scale=0.5]
    \draw (0, 0) grid (9, 9);
    \draw[very thick, scale=3] (0, 0) grid (3, 3);

   === \setcounter{row}{1}
    \setrow{}{}{}{}{}{}{}{}{}
    \setrow{}{}{}{}{}{}{}{}{}
    \setrow{}{}{}{}{}{}{}{}{}
	\setrow{}{}{}{}{}{}{}{}{}
	\setrow{}{}{}{}{}{}{}{}{}
	\setrow{}{}{}{}{}{}{}{}{}
	\setrow{}{}{}{}{}{}{}{}{}
	\setrow{}{}{}{}{}{}{}{}{}
	\setrow{}{}{}{}{}{}{}{}{}

	\node[anchor=center] at (-0.5, 0.5) {\tiny 00};
	\node[anchor=center] at (-0.5, 1.5) {\tiny 01};
	\node[anchor=center] at (-0.5, 2.5) {\tiny 02};
	\node[anchor=center] at (-0.5, 3.5) {\tiny 10};
	\node[anchor=center] at (-0.5, 4.5) {\tiny 11};
	\node[anchor=center] at (-0.5, 5.5) {\tiny 12};
	\node[anchor=center] at (-0.5, 6.5) {\tiny 20};
	\node[anchor=center] at (-0.5, 7.5) {\tiny 21};
	\node[anchor=center] at (-0.5, 8.5) {\tiny 22};

	\node[anchor=center] at (0.5, -0.5) {\tiny 00};
	\node[anchor=center] at (1.5, -0.5) {\tiny 01};
	\node[anchor=center] at (2.5, -0.5) {\tiny 02};
	\node[anchor=center] at (3.5, -0.5) {\tiny 10};
	\node[anchor=center] at (4.5, -0.5) {\tiny 11};
	\node[anchor=center] at (5.5, -0.5) {\tiny 12};
	\node[anchor=center] at (6.5, -0.5) {\tiny 20};
	\node[anchor=center] at (7.5, -0.5) {\tiny 21};
	\node[anchor=center] at (8.5, -0.5) {\tiny 22};

	\highlightrectangle{7}{1}{9}{3}
	\highlightrectangle{4}{4}{6}{6}
	\highlightrectangle{1}{7}{3}{9}
\end{tikzpicture}
}
\subfloat[$T_{0, I(2, 1/3)}$]{
\begin{tikzpicture}[scale=0.5]
    \draw (0, 0) grid (9, 9);
    \draw[very thick, scale=3] (0, 0) grid (3, 3);

    \setcounter{row}{1}
    \setrow{}{}{}{}{}{}{}{}{}
    \setrow{}{}{}{}{}{}{}{}{}
    \setrow{}{}{}{}{}{}{}{}{}
	\setrow{}{}{}{}{}{}{}{}{}
	\setrow{}{}{}{}{}{}{}{}{}
	\setrow{}{}{}{}{}{}{}{}{}
	\setrow{}{}{}{}{}{}{}{}{}
	\setrow{}{}{}{}{}{}{}{}{}
	\setrow{}{}{}{}{}{}{}{}{}

	\node[anchor=center] at (-0.5, 0.5) {\tiny 00};
	\node[anchor=center] at (-0.5, 1.5) {\tiny 01};
	\node[anchor=center] at (-0.5, 2.5) {\tiny 02};
	\node[anchor=center] at (-0.5, 3.5) {\tiny 10};
	\node[anchor=center] at (-0.5, 4.5) {\tiny 11};
	\node[anchor=center] at (-0.5, 5.5) {\tiny 12};
	\node[anchor=center] at (-0.5, 6.5) {\tiny 20};
	\node[anchor=center] at (-0.5, 7.5) {\tiny 21};
	\node[anchor=center] at (-0.5, 8.5) {\tiny 22};

	\node[anchor=center] at (0.5, -0.5) {\tiny 00};
	\node[anchor=center] at (1.5, -0.5) {\tiny 01};
	\node[anchor=center] at (2.5, -0.5) {\tiny 02};
	\node[anchor=center] at (3.5, -0.5) {\tiny 10};
	\node[anchor=center] at (4.5, -0.5) {\tiny 11};
	\node[anchor=center] at (5.5, -0.5) {\tiny 12};
	\node[anchor=center] at (6.5, -0.5) {\tiny 20};
	\node[anchor=center] at (7.5, -0.5) {\tiny 21};
	\node[anchor=center] at (8.5, -0.5) {\tiny 22};

    \highlightrectangle{7}{1}{9}{3}
	\highlightrectangle{1}{4}{3}{6}
	\highlightrectangle{4}{7}{6}{9}
\end{tikzpicture}
}
\caption{In $\mathbb{Q}_{3}$, the three possible $3 \times 9$ parallelograms containing the origin in $B(0, 9)$ are: $T_{0, I(0, 1/3)} := \{x \in \Q_{3}^2: |x_1|_{3} \leq 3, |x_2|_{3} \leq 9\}$, $T_{0, I(1, 1/3)} := \{x \in \Q_{3}^2: |x_{1} + 2x_2|_{3} \leq 3, |x_{2}|_{3} \leq 9\}$, and $T_{0, I(2, 1/3)} := \{x \in \Q_{3}^2: |x_{1} + 4x_2|_{3} \leq 3, |x_{2}|_{3} \leq 9\}$. Here for example, ``12" denotes the elements in $\Q_{3}$ given by
$1 \cdot 3^{-2} + 2 \cdot 3^{-1} + \cdots$.}
\end{figure}
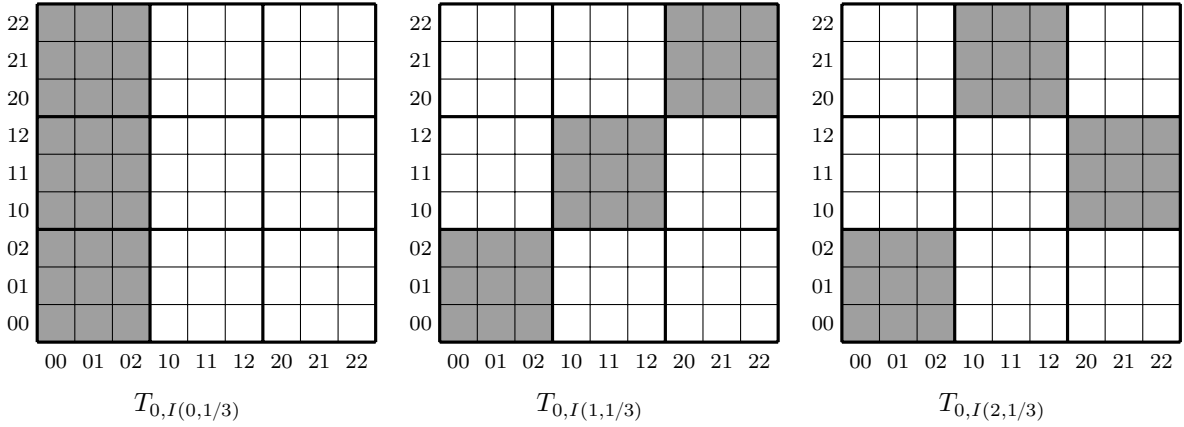
\end{center}
%
%
%
%
%
%
%
%
%
%
%
%
%
%
The translates of $T_{0, K}$ pointing in any direction completely partition any square of side length $\delta^{-2}$.
\begin{lemma}\label{tiling}
Let $\delta \in q^{-\N}$ and fix $K \in P_{\delta}$. Then
a $\delta^{-2} \times \delta^{-2}$ square in $\F^2$ can be partitioned into
$\delta^{-1}$ many translates $T_{0, K}$.
\end{lemma}
\begin{proof}
Let $B$ be the square of side length $\delta^{-2}$.
Without loss of generality, we may assume that $B = B(0, \delta^{-2})$.
Since
\begin{align*}
\{x \in \F: |x| \leq \delta^{-2}\} = \bigcup_{t \in \mc{S}}\{x \in \F: |x - t| \leq \delta^{-1}\}
\end{align*}
where $\mc{S} = \{\sum_{-2N \leq j < -N}a_{j}q^{j}: a_{j} \in \{0, 1, \ldots, q - 1\}\}$
and $\delta = q^{-N}$,
we can write
$B = \bigcup_{t}T_{\delta} + (t, 0).$
Applying $M_{a}^{-T}$ to both sides, we obtain that
$$B = M_{a}^{-T}B = \bigcup_{t} M_{a}^{-T}T_{\delta} + (t, 0) = \bigcup_{t}T_{0, K} + (t, 0)$$
where in the first equality we have used the ultrametric triangle inequality.
\end{proof}
\begin{center}
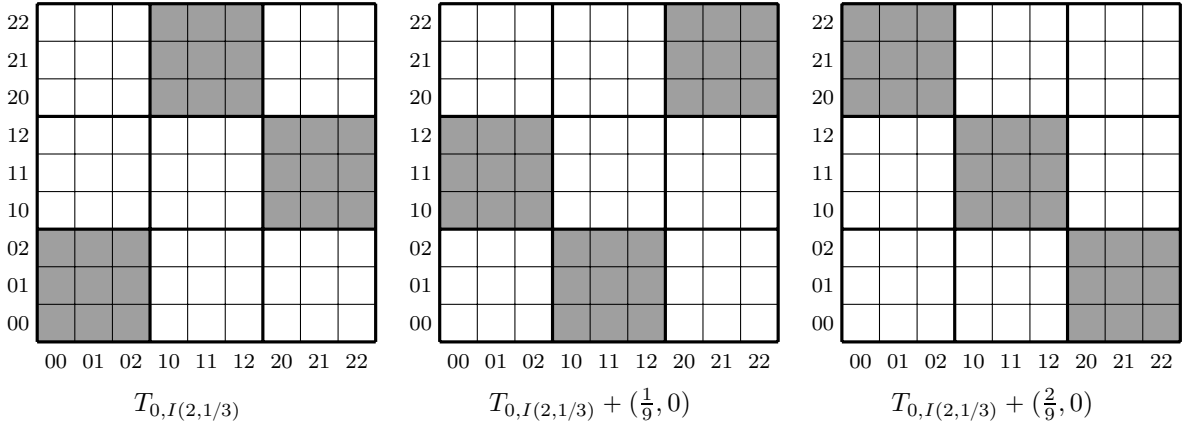
\begin{figure}[t!]
\captionsetup[subfigure]{labelformat=empty}
\centering
\subfloat[$T_{0, I(2, 1/3)}$]{
\begin{tikzpicture}[scale=0.5]
    \draw (0, 0) grid (9, 9);
    \draw[very thick, scale=3] (0, 0) grid (3, 3);

    \setcounter{row}{1}
    \setrow{}{}{}{}{}{}{}{}{}
    \setrow{}{}{}{}{}{}{}{}{}
    \setrow{}{}{}{}{}{}{}{}{}
	\setrow{}{}{}{}{}{}{}{}{}
	\setrow{}{}{}{}{}{}{}{}{}
	\setrow{}{}{}{}{}{}{}{}{}
	\setrow{}{}{}{}{}{}{}{}{}
	\setrow{}{}{}{}{}{}{}{}{}
	\setrow{}{}{}{}{}{}{}{}{}

	\node[anchor=center] at (-0.5, 0.5) {\tiny 00};
	\node[anchor=center] at (-0.5, 1.5) {\tiny 01};
	\node[anchor=center] at (-0.5, 2.5) {\tiny 02};
	\node[anchor=center] at (-0.5, 3.5) {\tiny 10};
	\node[anchor=center] at (-0.5, 4.5) {\tiny 11};
	\node[anchor=center] at (-0.5, 5.5) {\tiny 12};
	\node[anchor=center] at (-0.5, 6.5) {\tiny 20};
	\node[anchor=center] at (-0.5, 7.5) {\tiny 21};
	\node[anchor=center] at (-0.5, 8.5) {\tiny 22};

	\node[anchor=center] at (0.5, -0.5) {\tiny 00};
	\node[anchor=center] at (1.5, -0.5) {\tiny 01};
	\node[anchor=center] at (2.5, -0.5) {\tiny 02};
	\node[anchor=center] at (3.5, -0.5) {\tiny 10};
	\node[anchor=center] at (4.5, -0.5) {\tiny 11};
	\node[anchor=center] at (5.5, -0.5) {\tiny 12};
	\node[anchor=center] at (6.5, -0.5) {\tiny 20};
	\node[anchor=center] at (7.5, -0.5) {\tiny 21};
	\node[anchor=center] at (8.5, -0.5) {\tiny 22};

    \highlightrectangle{7}{1}{9}{3}
	\highlightrectangle{1}{4}{3}{6}
	\highlightrectangle{4}{7}{6}{9}
\end{tikzpicture}
}
\subfloat[$T_{0, I(2, 1/3)} + (\frac{1}{9}, 0)$]{
\begin{tikzpicture}[scale=0.5]
    \draw (0, 0) grid (9, 9);
    \draw[very thick, scale=3] (0, 0) grid (3, 3);

    \setcounter{row}{1}
    \setrow{}{}{}{}{}{}{}{}{}
    \setrow{}{}{}{}{}{}{}{}{}
    \setrow{}{}{}{}{}{}{}{}{}
	\setrow{}{}{}{}{}{}{}{}{}
	\setrow{}{}{}{}{}{}{}{}{}
	\setrow{}{}{}{}{}{}{}{}{}
	\setrow{}{}{}{}{}{}{}{}{}
	\setrow{}{}{}{}{}{}{}{}{}
	\setrow{}{}{}{}{}{}{}{}{}

	\node[anchor=center] at (-0.5, 0.5) {\tiny 00};
	\node[anchor=center] at (-0.5, 1.5) {\tiny 01};
	\node[anchor=center] at (-0.5, 2.5) {\tiny 02};
	\node[anchor=center] at (-0.5, 3.5) {\tiny 10};
	\node[anchor=center] at (-0.5, 4.5) {\tiny 11};
	\node[anchor=center] at (-0.5, 5.5) {\tiny 12};
	\node[anchor=center] at (-0.5, 6.5) {\tiny 20};
	\node[anchor=center] at (-0.5, 7.5) {\tiny 21};
	\node[anchor=center] at (-0.5, 8.5) {\tiny 22};

	\node[anchor=center] at (0.5, -0.5) {\tiny 00};
	\node[anchor=center] at (1.5, -0.5) {\tiny 01};
	\node[anchor=center] at (2.5, -0.5) {\tiny 02};
	\node[anchor=center] at (3.5, -0.5) {\tiny 10};
	\node[anchor=center] at (4.5, -0.5) {\tiny 11};
	\node[anchor=center] at (5.5, -0.5) {\tiny 12};
	\node[anchor=center] at (6.5, -0.5) {\tiny 20};
	\node[anchor=center] at (7.5, -0.5) {\tiny 21};
	\node[anchor=center] at (8.5, -0.5) {\tiny 22};

	\highlightrectangle{7}{4}{9}{6}
	\highlightrectangle{4}{1}{6}{3}
	\highlightrectangle{1}{7}{3}{9}
\end{tikzpicture}
}
\subfloat[$T_{0, I(2, 1/3)} + (\frac{2}{9}, 0)$]{
\begin{tikzpicture}[scale=0.5]
    \draw (0, 0) grid (9, 9);
    \draw[very thick, scale=3] (0, 0) grid (3, 3);

   === \setcounter{row}{1}
    \setrow{}{}{}{}{}{}{}{}{}
    \setrow{}{}{}{}{}{}{}{}{}
    \setrow{}{}{}{}{}{}{}{}{}
	\setrow{}{}{}{}{}{}{}{}{}
	\setrow{}{}{}{}{}{}{}{}{}
	\setrow{}{}{}{}{}{}{}{}{}
	\setrow{}{}{}{}{}{}{}{}{}
	\setrow{}{}{}{}{}{}{}{}{}
	\setrow{}{}{}{}{}{}{}{}{}

	\node[anchor=center] at (-0.5, 0.5) {\tiny 00};
	\node[anchor=center] at (-0.5, 1.5) {\tiny 01};
	\node[anchor=center] at (-0.5, 2.5) {\tiny 02};
	\node[anchor=center] at (-0.5, 3.5) {\tiny 10};
	\node[anchor=center] at (-0.5, 4.5) {\tiny 11};
	\node[anchor=center] at (-0.5, 5.5) {\tiny 12};
	\node[anchor=center] at (-0.5, 6.5) {\tiny 20};
	\node[anchor=center] at (-0.5, 7.5) {\tiny 21};
	\node[anchor=center] at (-0.5, 8.5) {\tiny 22};

	\node[anchor=center] at (0.5, -0.5) {\tiny 00};
	\node[anchor=center] at (1.5, -0.5) {\tiny 01};
	\node[anchor=center] at (2.5, -0.5) {\tiny 02};
	\node[anchor=center] at (3.5, -0.5) {\tiny 10};
	\node[anchor=center] at (4.5, -0.5) {\tiny 11};
	\node[anchor=center] at (5.5, -0.5) {\tiny 12};
	\node[anchor=center] at (6.5, -0.5) {\tiny 20};
	\node[anchor=center] at (7.5, -0.5) {\tiny 21};
	\node[anchor=center] at (8.5, -0.5) {\tiny 22};

	\highlightrectangle{1}{1}{3}{3}
	\highlightrectangle{4}{4}{6}{6}
	\highlightrectangle{7}{7}{9}{9}
\end{tikzpicture}
}
\caption{The three translates of $T_{0, I(2, 1/3)}$ that partition $B(0, 9)$. Informally, we are just shifting $T_{0, I(2, 1/3)} + (\frac{2}{9}, 0)$ and having the boundary of $B(0, 9)$ wrap around.}
\end{figure}
\end{center}
Lemma \ref{tiling} implies that we can tile $\F^{2}$
by translates of $T_{0, K}$ for each $K$. Let $\T(K)$ denote this tiling.
\begin{lemma}[Wavepacket decomposition]\label{wavepacketdecomposition}
Let $\delta \in q^{-\N}$ and fix $K \in P_{\delta}$.
Let $f$ be a Schwartz function with Fourier transform supported in $\ta_{K}$.
Then $|f|$ is constant on every $T \in \T(K)$
and $\wh{f1_{T}}$ is supported on $\ta_{K}$ for every $T \in \T(K)$.
Hence it is natural to write
\begin{align}\label{decomp}
f = \sum_{T \in \T(K)}f1_{T}
\end{align}
where each term $f1_{T}$ (which we will call a ``wavepacket") is Fourier supported
on $\ta_{K}$ and has constant modulus on every $T \in \T(K)$.
It also follows that if $\mc{T}$ is any subset of $\T(K)$, then
$\sum_{T \in \mc{T}}f1_{T}$ is Fourier supported in $\ta_K$.
\end{lemma}
\begin{rem}
In Lemma \ref{wavepacketdecomposition}, that $|f|$ is constant on every $T \in \mc{T}(K)$ is sometimes called the ``locally constant property."
\end{rem}
\begin{proof}
We will show that $|f(x)|$ is constant for all $x \in A + T_{0, K}$ for any
$A \in \F^2$. By Fourier inversion, we have that
\begin{align}
|f(x)| &= |\int_{\ta_K}\wh{f}(\xi)\chi(\xi \cdot x)\, d\xi|\nonumber\\
& = |\int_{|s| \leq \delta, |t| \leq \delta^2}\wh{f}(s + a, (s + a)^{2} + t)\chi((s + a)x_{1} + (s + a)^{2}x_{2})\chi(tx_2)\, ds\, dt|\nonumber\\
&= |\int_{|s| \leq \delta, |t| \leq \delta^2}\wh{f}(s + a, (s + a)^{2} + t)\chi(s(x_{1} + 2ax_2) + s^{2}x_2)\chi(tx_2)\, ds\, dt|.\label{fconst}
\end{align}
For $x \in A + T_{0, K}$, write $x = A + M_{a}^{-T}y$ for $y \in T_{\delta}$.
Then $M_{a}^{T}x = M_{a}^{T}A + y$ and hence
\begin{align*}
\chi(s(x_{1} + 2ax_2) + s^{2}x_2)= \chi((s, s^2) \cdot M_{a}^{T}A)\chi((s, s^2)\cdot y) = \chi((s, s^2) \cdot M_{a}^{T}A)
\end{align*}
since $sy_{1}, s^{2}y_{2} \in \O$. 
Next $\chi(tx_2) = \chi(tA_{2})\chi(ty_{2}) = \chi(tA_{2})$
since $ty_{2} \in \O$.
Therefore \eqref{fconst} is equal to
\begin{align*}
|\int_{|s| \leq \delta, |t| \leq \delta^2}\wh{f}(s + a, (s + a)^{2} + t)\chi((s, s^2) \cdot M_{a}^{T}A)\chi(tA_{2})\, ds\, dt|
\end{align*}
which is a constant independent of $x$ and only depends on $A$, $K$, and $f$.

Next to prove $\wh{f1_{T}}$ is supported on $\ta_K$, it suffices to observe
that $\wh{f1_{T}}  = \wh{f} \ast \wh{1_{T}}$. Observe that
$\wh{1_{T}}$ is a modulation of $\wh{1_{T_{0, K}}}$ and
\begin{align*}
\wh{1_{T_{0, K}}} = \int_{M_{a}^{-T}T_{\delta}}\chi(-x \cdot \xi)\, dx = \int_{T_{\delta}}\chi(-M_{a}^{-T}y \cdot \xi)\,dy = \delta^{3}1_{\ta_{\delta}}(M_{a}^{-1}\xi).
\end{align*} 
Thus $\wh{1_T}$ is supported on $M_{a}\ta_{\delta} = \ta_{K} - \ta_{K}$.

Finally, the decomposition \eqref{decomp} follows since translates of $T_{0, K}$
tiles $\F^2$. This completes the proof.
\end{proof}
\begin{cor}\label{constcor}
If $K \in P_{\delta}$ and $f$ is Fourier supported on $\ta_K$,
then $|f|$ is constant on any ball of radius $\delta$.
\end{cor}
\begin{proof}
Observe that $B(a, \delta) \subset a + T_{0, K}$
and $|f|$ is constant on $a + T_{0, K}$ and hence also on $B(a, \delta)$.
\end{proof}

We now end with a lemma about the intersection of two non-parallel tubes (of course we only use this bound when they intersect).
\begin{lemma}\label{intersectionlemma}
Let $\delta \in q^{-\N}$ and $K, K' \in P_{\delta}$ be two distinct intervals. 
Consider a translate $T$
of $T_{0, K}$ and a translate $T'$ of $T_{0, K'}$.
Then $|T \cap T'| \leq \delta^{-2}/d(K, K')$.
\end{lemma}
\begin{proof}
Write $T = T_{0, K} + (r, s)$ and $T' = T_{0, K'} + (r', s')$.
Without loss of generality, we may assume that $(r, s) = (0, 0)$.
Applying $M_{a}^{T}$ we may assume that $T = T_{\delta}$.
Then for some $b \in K'$,
\begin{align*}
T \cap T' &\subset \{x \in \F^2: |x_1| \leq \delta^{-1}, |x_1 + 2bx_2 - r'| \leq \delta^{-1}\}\\
&\subset \{x \in \F^2: |x_1| \leq \delta^{-1}, |x_2 - \frac{r'}{2b}| \leq \frac{\delta^{-1}}{|b|}\}
\end{align*}
which has area $\leq \delta^{-2}/|b| = \delta^{-2}/d(K, K')$.
\end{proof}

\section{Basics of decoupling}\label{basics}
For ease of notation, we will now let $D_{p}(\delta) := D_{p, \F}(\delta)$.

\subsection{Sharpness of Theorem \ref{main}}
Before we start our proof of Theorem \ref{main}, we first prove that up to a $\delta^{-\vep}$ factor, the estimate is sharp. 
We first start with a computation.
\begin{lemma}\label{takfi}
We have
\begin{align*}
\widecheck{1}_{\ta_K}(x) = \chi(ax_{1} + a^{2}x_{2})\delta^{3}1_{T_{0, K}}
\end{align*}
for any $a \in K$.
\end{lemma}
\begin{rem}
The right hand side is well-defined since for $x \in T_{0, K}$,
$$\chi((a - b)x_{1} + (a^{2} - b^{2})x_{2}) = \chi((a - b)(x_{1} + 2ax_2) - (a - b)^{2}x_2) = 1$$
if $a, b \in K$ (and so $|a - b| \leq \delta$).
\end{rem}
\begin{proof}
We compute
\begin{align*}
\widecheck{1}_{\ta_K} &= \int_{\ta_K}\chi(x \cdot \xi)\, d\xi = \int_{|s| \leq \delta, |t| \leq \delta^2}e((s + a)x_{1} + (s + a)^{2}x_{2})e(tx_2)\, ds\, dt\\
&= \chi(ax_{1} + a^{2}x_2)\int_{|t| \leq \delta^{2}}\chi(tx_2)\, dt\int_{|s| \leq \delta}\chi(s(x_1 + 2ax_2) + s^{2}x_2)\, ds.
\end{align*}
By Proposition \ref{ballft}, the 
integral in $t$ is equal to $\delta^{2}1_{|x_2| \leq \delta^{-2}}$
and the integral in $s$ is equal to $\delta 1_{|x_{1} + 2ax_2| \leq \delta^{-1}}$.
This completes the proof of the lemma.
\end{proof}
This then allows us to show the following result.
\begin{prop}\label{lowerbound}
We have $D_{p}(\delta) \gtrsim 1 + \delta^{-(\frac{1}{2} - \frac{3}{p})}$.
\end{prop}
\begin{proof}
To see that $D_{p}(\delta) \geq 1$, take $f = \widecheck{1}_{\ta_{\delta}}$.
Then $$\nms{f}_{L^{p}(\F^2)} = \nms{\widecheck{1}_{\ta_{\delta}}}_{L^{p}(\F^2)} = (\sum_{K \in P_{\delta}}\nms{f_{K}}_{L^{p}(\F^2)}^{2})^{1/2}$$
and hence
$$D_{p}(\delta) \geq \frac{\nms{f}_{L^{p}(\F^2)}}{(\sum_{K \in P_{\delta}}\nms{f_K}_{L^{p}(\F^2)}^{2})^{1/2}} = 1.$$

Next, to see that $D_{p}(\delta) \geq \delta^{-(\frac{1}{2} - \frac{3}{p})}$, take
$f = \sum_{K \in P_{\delta}}f_{K}$ where $f_K$ is such that $\wh{f_K} = \delta^{-3}1_{\ta_{K}}$. Therefore $f_{K} = \chi(a_{K}x_1 + a_{K}^{2}x_2)1_{T_{0, K}}$ 
for any $a_K \in K$. This implies that
$\nms{f_{K}}_{L^{p}(\F^2)} = \delta^{-3/p}$ and so
$$(\sum_{K \in P_{\delta}}\nms{f_K}_{L^{p}(\F^2)}^{2})^{1/2} = \delta^{-1/2 - 3/p}.$$
We also have
\begin{align*}
f = \sum_{K \in P_{\delta}}f_K = \sum_{K \in P_{\delta}}\chi(a_{K}x_{1} + a_{K}^{2}x_2)1_{T_{0, K}}.
\end{align*}
For $x \in \O^{2}$, since $a_{K} \in K \subset \O$, $a_{K}x_1 + a_{K}^{2}x_2 \in \O$
and hence $\chi(a_{K}x_1 + a_{K}^{2}x_2) = 1$ for $x \in \O^2$.
Furthermore, $\O^{2} \subset T_{0, K}$ for every $K \in P_{\delta}$.
Thus
\begin{align*}
\nms{f}_{L^{p}(\F^2)} \geq \nms{f}_{L^{p}(\O^2)} \geq \delta^{-1}
\end{align*}
and hence
\begin{align*}
D_{p}(\delta) \geq \frac{\nms{f}_{L^{p}(\F^2)}}{(\sum_{K \in P_{\delta}}\nms{f_K}_{L^{p}(\F^2)}^{2})^{1/2}} \geq \delta^{-1-(-\frac{1}{2} - \frac{3}{p})} = \delta^{-(\frac{1}{2} - \frac{3}{p})}.
\end{align*}
\end{proof}

\subsection{Interpolation of decoupling estimates}
Now we reduce Theorem \ref{main} to just showing the case when $p = 6$. To do this, we need to be able to interpolate decoupling estimates.
First we need that Fourier restriction to $\ta_K$ is a bounded operator on $L^p$.
\begin{lemma}\label{ftakbdd}
Let $K \in P_{\delta}$ and $f$ be Fourier supported in $\bigcup_{K \in P_{\delta}}\ta_K$.
Then $$\nms{f_{K}}_{L^{p}(\F^2)} \leq \nms{f}_{L^{p}(\F^2)}$$
for all $1 \leq p \leq \infty$.
\end{lemma}
\begin{proof}
By definition, $f_{K} = f \ast \widecheck{1}_{\ta_K}$ and hence Young's inequality
gives that
\begin{align*}
\nms{f_K}_{L^{p}(\F^2)} = \nms{f \ast \widecheck{1}_{\ta_K}}_{L^{p}(\F^2)} \leq \nms{f}_{L^{p}(\F^2)}\nms{\widecheck{1}_{\ta_K}}_{L^{1}(\F^2)}.
\end{align*}
Applying Lemma \ref{takfi} shows that $\nms{\widecheck{1}_{\ta_K}}_{L^{1}(\F^2)} = 1$ 
which completes the proof of the lemma.
\end{proof}

Now we are ready to show that one can interpolate decoupling estimates.
\begin{prop}\label{interpolate}
If $\frac{1}{p} = \frac{\ta}{p_0} + \frac{1-\ta}{p_1}$, then
$$D_{p}(\delta) \leq D_{p_0}(\delta)^{\ta}D_{p_1}(\delta)^{1 - \ta}.$$
\end{prop}
\begin{proof}
Write $P_{\delta} := \{I_1, \ldots, I_N\}$ for some $N = \delta^{-1} \in q^{\N}$.
Define the operator $$T[(f_{j})_{j = 1}^{N}] := \sum_{j = 1}^{N}P_{\ta_{I_j}}f_j$$
where $P_{\ta_{I_j}}f := f \ast \widecheck{1}_{\ta_{I_j}}$ is the Fourier projection of $f$ to $\ta_{I_j}$.
Then for $r = p_0, p_1$, we have
\begin{align*}
\nms{T[(f_j)]}_{L^{r}(\F^2)} &\leq D_{r}(\delta)(\sum_{k = 1}^{N}\nms{(\sum_{j = 1}^{N}P_{\ta_{I_j}}f_j)_{I_k}}_{L^{r}(\F^2)}^{2})^{1/2} = D_{r}(\delta)(\sum_{k = 1}^{N}\nms{P_{\ta_{I_k}}f_k}_{L^{r}(\F^2)}^{2})^{1/2}\\
& \leq D_{r}(\delta)(\sum_{k = 1}^{N}\nms{f_k}_{L^{r}(\F^2)}^{2})^{1/2} = D_{r}(\delta)\nms{(f_j)}_{\ell^{2}L^{r}(\F^2)}
\end{align*}
where in the second inequality we have used Lemma \ref{ftakbdd}.
Therefore the operator $T$ sends $\ell^{2}(L^{p_0}(\F^2))$ into $L^{p_0}(\F^2)$ and also $\ell^{2}(L^{p_1}(\F^2))$ into $L^{p_1}(\F^2)$.
Thus by Riesz-Thorin interpolation, $T: \ell^{2}(L^{p}(\F^2)) \rightarrow L^{p}(\F^2)$ for all $p \in [p_0, p_1]$
and the operator norm of $T$, that is, $D_{p}(\delta)$, is $\leq D_{p_0}(\delta)^{\ta}D_{p_1}(\delta)^{1 - \ta}$.
\end{proof}

Therefore if we could show that $D_{6}(\delta) \lsm_{\vep} \delta^{-\vep}$, then interpolating 
the estimate that $D_{2}(\delta) = 1$, which follows from Plancherel,
and $D_{\infty}(\delta) \leq \delta^{-1/2}$ which follows from Cauchy-Schwarz proves
Theorem \ref{main}. Thus it remains to just show the $l^{2}L^{6}$ decoupling estimate.

\subsection{Parabola decoupling over $\F$ implies quadratic VMVT}
Before we begin our proof of $l^{2}L^{6}$ parabola decoupling, we show that such a decoupling theorem over $\F$ is strong enough to imply
exponential sum estimates.
\begin{lemma}
Suppose for some odd prime $q$, we knew that
\begin{align}\label{l6dec}
\nms{f}_{L^{6}(\F^2)} \lsm_{q, \vep} \delta^{-\vep}(\sum_{J \in P_{\delta}}\nms{f_{J}}_{L^{6}(\F^2)}^{2})^{1/2}
\end{align}
for all $f$ Fourier supported in $\bigcup_{K \in P_{\delta}}\ta_K$.
Then $K(N) \lsm_{q, \vep} N^{\vep}$.
\end{lemma}
\begin{proof}
By zeroing out entries, $K(N)$ is increasing in $N$. 
Thus it suffices to only show the desired estimate
in the case when $N = q^{t}$ for some $t \in \N$. 
Indeed, suppose we knew this for all $N \in q^{\N}$.
Then for $N \not\in q^{\N}$, we can find an $s$ such that $q^{s} < N < q^{s + 1}$
and hence for some $C_{\vep}$, 
$K(N) \leq K(q^{s + 1}) \leq C_{\vep, q} (q^{s + 1})^{\vep} = C_{\vep, q}q^{\vep}q^{s\vep} \leq C_{\vep, q}q^{\vep}N^{\vep}.$

Choose $\delta = q^{-t}$ and $F$ such that
\begin{align*}
\wh{F}(\xi, \eta) = \sum_{n = 1}^{q^t}q^{20t}a_{n}1_{(n, n^2) + B(0, q^{-10t})}(\xi, \eta).
\end{align*}
We first check that $\wh{F}$ is supported in $\bigcup_{K \in P_{q^{-t}}}\ta_K$.
It suffices to show that for each $K = I(n, q^{-t})$, we have 
$(n, n^2) + B(0, q^{-10t}) \subset \ta_{K}$. Indeed if $|\xi - n| \leq q^{-10t}$
and $|\eta - n^2| \leq q^{-10t}$, then $\xi \in K$ and
\begin{align*}
|\xi^2 - \eta| = |(\xi - n)^2 + 2n(\xi - n) + n^{2} - \eta| \leq \max(|\xi - n|^2, |2n||\xi -n|, |n^2 - \eta|). 
\end{align*}
Since $q \geq 3$ is an odd prime, $|2n| \leq 1$ and so the above is $\leq q^{-10t} \leq q^{-2t}$.

Inverting the Fourier transform gives that
\begin{align*}
F(x) = (\sum_{n = 1}^{q^t}a_{n}\chi(x_1 n + x_{2}n^2))1_{B(0, q^{10t})}(x)
\end{align*}
and $F_{K} = a_{n}\chi(x_1 n + x_{2}n^2)1_{B(0, q^{10t})}(x)$ with $K = I(n, q^{-t})$.
This gives that $\nms{F_K}_{L^{6}(\F^2)} = q^{20t/6}|a_{n}|$ and so
\begin{align}\label{rhs}
(\sum_{K \in P_{\delta}}\nms{F_K}_{L^{6}(\F^2)}^{2})^{3} = q^{20t}(\sum_{n = 1}^{q^t}|a_n|^{2})^{3}.
\end{align}
On the other hand, expanding yields
\begin{align*}
&\nms{F}_{L^{6}(\F^2)}^{6} = \sum_{n_1, \ldots, n_6 = 1}^{q^t}a_{n_1}a_{n_2}a_{n_3}\ov{a_{n_4}a_{n_5}a_{n_6}}\times\\
&\int_{B(0, q^{10t})}\chi((n_1 + n_2 + n_3 - n_4 - n_5 - n_6)x_1 + (n_{1}^{2} + n_{2}^{2} + n_{3}^{2} - n_{4}^{2} - n_{5}^{2} - n_{6}^{2})x_2)\, dx
\end{align*}
and applying Proposition \ref{ballft} gives that the above is equal to
\begin{align}
\begin{aligned}\label{fl6}
 q^{20t}\sum_{n_1, \ldots, n_6 = 1}^{q^t}a_{n_1}a_{n_2}a_{n_3}\ov{a_{n_4}a_{n_5}a_{n_6}}1_{|n_1 + n_2 + n_3 - n_4 - n_5 - n_6| \leq q^{-10t}, |n_{1}^{2} + n_{2}^{2} + n_{3}^{2} -n_{4}^{2} - n_{5}^{2} - n_{6}^{2}| \leq q^{-10t}}.
\end{aligned}
\end{align}
Note that $(n_1, \ldots, n_6) \in ([1, q^t] \cap \N)^{6}$ satisfies
\begin{align}
\begin{aligned}\label{congruence}
n_{1} + n_2 + n_3 - n_4 - n_5 - n_6 &\equiv 0 \Mod{q^{10t}}\\
n_{1}^{2} + n_{2}^{2} + n_{3}^{2} - n_{4}^{2} - n_{5}^{2} - n_{6}^{2} &\equiv 0 \Mod{q^{10t}}
\end{aligned}
\end{align}
if and only if it satisfies
\begin{align}\label{solution}
\begin{aligned}
n_{1} + n_2 + n_3 - n_4 - n_5 - n_6 &= 0\\
n_{1}^{2} + n_{2}^{2} + n_{3}^{2} - n_{4}^{2} - n_{5}^{2} - n_{6}^{2} &= 0.
\end{aligned}
\end{align}
Indeed, since $n_{i} \in [1, q^t] \cap \N$, we have $n_{1} + n_2 + n_3 - n_4 - n_5 - n_6 \in [-3q^t, 3q^t] \cap \N$ and 
$n_{1}^{2} + n_{2}^{2} + n_{3}^{2} - n_{4}^{2} - n_{5}^{2} - n_{6}^{2} \in [-3q^{2t}, 3q^{2t}] \cap \N$.
Since $3q^{t} \leq 3q^{2t} \leq q^{10t}$, the only integer $\equiv 0 \Mod{q^{10t}}$ in $[-3q^t, 3q^t]$ and $[-3q^{2t}, 3q^{2t}]$
is 0 itself.
Thus \eqref{fl6} is equal to
\begin{align*}
q^{20t}\nms{\sum_{n = 1}^{q^t}a_{n}e(nx_1 + n^{2}x_2)}_{L^{6}(\T^2)}^{6}.
\end{align*}
Combining this with \eqref{rhs} and \eqref{l6dec}, then shows that 
$K(N) \lsm_{q, \vep} N^{\vep}$ as desired.
\end{proof}
\begin{rem}
The equivalence of \eqref{congruence} with \eqref{solution} for $n_{i} \in [1, q^t] \cap \N$ is the same equivalence
used in efficient congruencing, see for example \cite[Eq. (2.9)]{WooleyNested}.
\end{rem}

\section{A proof of $l^{2}L^{6}$ parabola decoupling}\label{l2l6pf}
Our main goal now is to prove that $D_{6}(\delta) \lsm_{\vep} \delta^{-\vep}$. 
For ease of notation, we will let $D(\delta) := D_{6}(\delta) = D_{6, \F}(\delta)$. 

\subsection{Parabolic rescaling}
\begin{prop}
Let $f$ be a Schwartz function on $\F^2$ Fourier supported in $\bigcup_{K \in P_{\delta}}\ta_K$.
Then for any interval $I \subset \O$ of length $\sigma \geq \delta$, we have
\begin{align*}
\nms{f_{I}}_{L^{p}(\F^2)} \leq D_{p}(\frac{\delta}{\sigma})(\sum_{K \in P_{\delta}(I)}\nms{f_K}_{L^{p}(\F^2)}^{2})^{1/2}.
\end{align*} 
\end{prop}
\begin{proof}
Write $I = I(a, \sigma)$ for some $a \in I$. We have
\begin{align*}
|f_{I}(x)| = |\int_{|t| \leq \delta^{2}}\int_{I(a, \sigma)}\wh{f}(s, s^2 + t)\chi(sx_1 + s^{2}x_2)e(tx_2)\, ds\, dt|.
\end{align*}
Let $s' := \sigma^{-1}(s - a)$ and $t' := \sigma^{-2}t$. Then the above is equal to
\begin{align*}
|\sigma^{3}\int_{|t'| \leq (\delta/\sigma)^{2}}\int_{\O}\wh{f}(\sigma s' + a, (\sigma s' + a)^{2} + \sigma^{2}t')\chi(\sigma s'(x_1 + 2ax_2) + \sigma^{2}s'^{2}x_2)\chi(\sigma^{2}t' x_2)\, ds'\, dt'|.
\end{align*}
Write $T_{a, \sigma} = (\begin{smallmatrix} \sigma & 2a\sigma\\0 & \sigma^2\end{smallmatrix})$,
$f_{a}(x) := f(x)\chi(ax_1 + a^{2} x_2)$, and $g_{a}(y) = \sigma^{-3}f_{a}(T_{a, \sigma}^{-1}y)$, 
then $\wh{g_{a}}(s', s'^{2} + t') = \wh{f}(\sigma s' + a, (\sigma s' + a)^{2} + \sigma^{2}t')$ and
$g_{a}$ is Fourier supported in $\bigcup_{K' \in P_{\delta/\sigma}}\ta_{K'}$. Therefore
$\nms{f_{I}}_{L^{p}(\F^2)} = \sigma^{3 - 3/p}\nms{g_{a}}_{L^{p}(\F^2)}$.
Applying the definition of $D_{p}(\delta/\sigma)$ and undoing the above change of variables finishes the proof.
\end{proof}

\subsection{Bilinear reduction}
Let $\nu \in q^{-\N}$ be an intermediate scale chosen later such that $\delta \ll \nu \ll 1$.
On a first read, it may be useful to set $\nu = 1/q$ though we will later choose
it to be something larger, see the proof of Proposition \ref{iter3}.
Suppose $a$ and $b$ were integers such that $\delta \leq \min(\nu^a, \nu^b)$.
Let $M_{a, b}(\delta, \nu)$
be the best constant such that
\begin{align}\label{mabdef}
\int_{\F^2}|f_{I}|^{2}|f_{I'}|^{4} \leq M_{a, b}(\delta, \nu)^{6}(\sum_{K \in P_{\delta}(I)}\nms{f_K}_{L^{6}(\F^2)}^{2})(\sum_{K' \in P_{\delta}(I')}\nms{f_{K'}}_{L^{6}(\F^2)}^{2})^{2}
\end{align}
for all $f$ Fourier supported in $\bigcup_{K \in P_{\delta}}\ta_K$ and all intervals
$I, I' \subset \O$ with $|I| = \nu^a$, $|I'| = \nu^b$ and $d(I, I') > \nu$.

If we think of $\nu \approx \delta^{\vep}$ and $a$ and $b$
are small, then we essentially have $M_{a, b}(\delta, \nu) \approx D(\delta)$
as $I$ and $I'$ are rather large intervals and we have not
decoupled much. However, if $a$ and $b$
were so large such that $\nu^a$ and $\nu^b$ are basically the same size
as $\delta$, then $M_{a, b}(\delta, \nu) \approx 1$
after applying the triangle inequality.
Thus our goal will be to pass from small $(a, b)$ to large $(a, b)$ efficiently.

The following trivial bound will be useful when we reach large $(a, b)$.
\begin{lemma}\label{mabtriv}
If $\delta \leq \min(\nu^a, \nu^b)$, then
$$M_{a, b}(\delta, \nu) \leq D(\frac{\delta}{\nu^a})^{1/3}D(\frac{\delta}{\nu^b})^{2/3}.$$
\end{lemma}
\begin{proof}
This follows immediately from $\int f^{2}g^4 \leq (\int f^6)^{1/3}(\int g^6)^{2/3}$, parabolic rescaling, and
applying the definitions of $D(\delta)$ and $M_{a, b}(\delta, \nu)$.
\end{proof}

The previous lemma also implies that if we knew a good bound for $D(\delta)$, then we would know a good
bound for $M_{a, b}(\delta, \nu)$. However, what we want is the opposite. This is accomplished in the next
proposition and essentially reduces the linear decoupling problem to a bilinear decoupling problem.
\begin{prop}\label{bired}
If $\delta \leq \nu$, then
$$D(\delta) \lsm D(\frac{\delta}{\nu}) + \nu^{-O(1)}M_{1, 1}(\delta, \nu).$$
\end{prop}
\begin{proof}
Fix arbitrary $f$ Fourier supported in $\bigcup_{K \in P_{\delta}}\ta_K$.
We have
\begin{align*}
\int_{\F^2}|f|^{6} &= \int_{\F^2}|\sum_{J, J' \in P_{\nu}}f_{J}f_{J'}|^{3} \lsm \int_{\F^2}|\sum_{J \in P_{\nu}}f_{J}^{2}|^{3} + \int_{\F^2}|\sum_{\st{J, J' \in P_{\nu}\\d(J, J') > \nu}}f_{J}f_{J'}|^{3}.
\end{align*}
By Minkowski's inequality and parabolic rescaling, the first term above is
\begin{align}\label{firstbil}
\leq (\sum_{J \in P_{\nu}}\nms{f_{J}}_{L^{6}(\F^2)}^{2})^{\frac{1}{2} \cdot 6} \leq D(\frac{\delta}{\nu})^{6}(\sum_{K \in P_{\delta}}\nms{f_K}_{L^{6}(\F^2)}^{2})^{\frac{1}{2} \cdot 6}.
\end{align}
On the other hand, the second term is
\begin{align}\label{secondbil}
\leq \nu^{-O(1)}\max_{\st{J, J' \in P_{\nu}\\d(J, J') > \nu}}\int_{\F^2}|f_{J}|^{3}|f_{J'}|^{3}.
\end{align}
Observing that $\int f^{3}g^{3} \leq (\int f^{2}g^{4})^{1/2}(\int f^{4}g^{2})^{1/2}$ and then applying the definition of
$M_{1, 1}(\delta, \nu)$ to each term then shows that \eqref{secondbil} is
$$\leq \nu^{-O(1)}M_{1, 1}(\delta, \nu)^{6}(\sum_{K \in P_{\delta}}\nms{f_K}_{L^{6}(\F^2)}^{2})^{3}.$$
Combining this with \eqref{firstbil} and applying the definition of $D(\delta)$ completes the proof.
\end{proof}
As our argument is bilinear in nature, we could use H\"{o}lder's inequality to reduce to the bilinear decoupling problem.
Before we continue, we give a second proof of Proposition \ref{bired} that relies on the broad/narrow decomposition first introduced
by Bourgain and Guth in \cite{BG11}. The broad/narrow decomposition is more commonly used to reduce linear decoupling problems
to multilinear decoupling problems.
Let $f$ be Fourier supported in $\bigcup_{K \in P_{\delta}}\ta_K$. For each $x_0 \in \F^2$,
let $\mc{I}_{x_0}$ be the set of all intervals $J' \in P_{\nu}$ such that
$|f_{J'}(x_0)| \geq \nu \max_{J \in P_{\nu}}|f_{J}(x_0)|$. We have two cases: $\# \mc{I}_{x_0} \geq 2$
or $\#\mc{I}_{x_0} < 2$.

If $\#\mc{I}_{x_0} \geq 2$, let $J_{1}$ and $J_{2}$ be two distinct
intervals in $\mc{I}_{x_0}$. Note that the choice of $J_1$ and $J_2$ depends
on $x_0$. We then have
\begin{align*}
|f(x_0)| \leq \sum_{J \in P_{\nu}}|f_{J}(x_0)| &\leq \nu^{-1}\max_{J \in P_{\nu}}|f_{J}(x_0)|\\
& \leq \nu^{-2}|f_{J_1}(x_0)|^{1/2}|f_{J_2}(x_0)|^{1/2} \leq \nu^{-2}\max_{\st{J, J' \in P_{\nu}\\d(J, J') > \nu}}|f_{J}(x_0)f_{J'}(x_0)|^{1/2}
\end{align*}
where in the third inequality we have used that $J_1, J_2 \in \mc{I}_{x_0}$ and in the last inequality
we have used that $J_1$ and $J_2$ are distinct and so $d(J_1, J_2) \geq q\nu$ by Lemma \ref{lengthpart}.
On the other hand if $\#\mc{I}_{x_0} < 2$, then we have
\begin{align*}
|f(x_0)| \leq |\sum_{J \in \mc{I}_{x_0}}f_{J}(x_0)| + |\sum_{J \not\in \mc{I}_{x_0}}f_{J}(x_0)| < 2\max_{J \in P_{\nu}}|f_{J}(x_0)| + \sum_{J \not\in \mc{I}_{x_0}}|f_{J}(x_0)| \leq 3\max_{J \in P_{\nu}}|f_{J}(x_0)|.
\end{align*}
Therefore we have shown that for each $x \in \F^2$, we have
\begin{align*}
|f(x)| \leq 3\max_{J \in P_{\nu}}|f_{J}(x)| + \nu^{-2}\max_{\st{J, J' \in P_{\nu}\\d(J, J') > \nu}}|f_{J}(x)f_{J'}(x)|^{1/2}.
\end{align*}
Raising both sides to the 6th power and integrating gives
\begin{align*}
\int_{\F^2}|f|^{6} &\lsm \int_{\F^2}\max_{J \in P_{\nu}}|f_{J}(x)|^{6}\, dx + \nu^{-O(1)}\int_{\F^2}\max_{\st{J, J' \in P_{\nu}\\d(J, J') > \nu}}|f_{J}(x)f_{J'}(x)|^{3}\\
&\lsm \int_{\F^2}(\sum_{J \in P_{\nu}}|f_{J}(x)|^{2})^{3} + \nu^{-O(1)}\max_{\st{J, J' \in P_{\nu}\\d(J, J') > \nu}}\int_{\F^2}|f_{J}(x)f_{J'}(x)|^{3}.
\end{align*} 
Now proceeding as in \eqref{firstbil} and \eqref{secondbil}
then finishes the proof.

\subsection{The key step}
The following lemma turns out to be the efficient way to climb up from small $(a, b)$ to large $(a, b)$.
\begin{prop}\label{keystep1}
If $1 \leq a \leq 2b$ and $\delta \leq \nu^{2b}$, then $$M_{a, b}(\delta, \nu) \lsm \nu^{-O(1)}M_{2b, b}(\delta, \nu).$$
\end{prop}
To prove Proposition \ref{keystep1}, it suffices to show the following lemma.
\begin{lemma}\label{keystep2}
If $1 \leq a \leq 2b$, $|I| = \nu^a$, $|I'| = \nu^b$, $d(I, I') \geq q \nu$, and $\delta \leq \nu^{2b}$ we have
\begin{align*}
\int_{\F^2}|f_{I}|^{2}|f_{I'}|^{4} \lsm \nu^{-O(1)}\sum_{J \in P_{\nu^{2b}}(I)}\int_{\F^2}|f_{J}|^{2}|f_{I'}|^{4}
\end{align*}
for all $f$ Fourier supported in $\bigcup_{K \in P_{\delta}}\ta_{K}$.
\end{lemma}
Indeed, Proposition \ref{keystep1} follows from Lemma \ref{keystep2} by applying the definition of $M_{2b, b}(\delta, \nu)$ to
$\int_{\F^2}|f_{J}|^{2}|f_{I'}|^{4}$ and then applying the definition of $M_{a, b}(\delta, \nu)$.
Next, by the proof of parabolic rescaling and shifting $I'$ to the origin, to prove Lemma \ref{keystep2} it suffices to instead prove the following result.
\begin{lemma}\label{keystep3}
If $1 \leq a \leq 2b$, $|I| = \nu^a$, $d(I, 0) \geq q \nu$, and $\delta \leq \nu^{2b}$ we have
\begin{align}\label{ks3}
\int_{\F^2}|f_{I}|^{2}|f_{I(0, \nu^b)}|^{4} \lsm \nu^{-O(1)}\sum_{J \in P_{\nu^{2b}}(I)}\int_{\F^2}|f_{J}|^{2}|f_{I(0, \nu^b)}|^{4}
\end{align}
for all $f$ Fourier supported in $\bigcup_{K \in P_{\delta}}\ta_{K}$.
\end{lemma}
We will give five somewhat different proofs of Lemma \ref{keystep3} later in Section \ref{fiveproofs}. In some of the proofs we will prove instead the following reduction.
\begin{lemma}\label{keystep4}
If $1 \leq a \leq 2b$, $|I|= \nu^a$, $d(I, 0) \geq q \nu$, and $\delta \leq \nu^{2b}$, we have
\begin{align}\label{ks4}
\int_{B(0, \nu^{-2b})}|f_{I}|^{2}|f_{I(0, \nu^b)}|^{4} \lsm \nu^{-O(1)}\sum_{J \in P_{\nu^{2b}}(I)}\int_{B(0, \nu^{-2b})}|f_{J}|^{2}|f_{I(0, \nu^b)}|^{4}
\end{align}
for all $f$ Fourier supported in $\bigcup_{K \in P_{\delta}}\ta_{K}$.
\end{lemma}
To see that Lemma \ref{keystep4} implies Lemma \ref{keystep3} (and hence Proposition \ref{keystep1}), since \eqref{ks4} is true for all $f$ Fourier supported 
in $\bigcup_{K \in P_{\delta}}\ta_K$, by translation we have that
\begin{align}\label{ks5}
\int_{B}|f_{I}|^{2}|f_{I(0, \nu^b)}|^{4} \lsm \nu^{-O(1)}\sum_{J \in P_{\nu^{2b}}(I)}\int_{B}|f_{J}|^{2}|f_{I(0, \nu^b)}|^{4}
\end{align}
for all squares $B$ of side length $\nu^{-2b}$ and all $f$ Fourier supported in $\bigcup_{K \in P_{\delta}}\ta_K$. 
Since squares of side length $\nu^{-2b}$ tile $\F^2$, summing \eqref{ks5} over $B$ then gives Lemma \ref{keystep3}.

\subsection{Iteration}
Before we give its proof, 
we now show that if we had Lemma \ref{keystep3} (and hence Proposition \ref{keystep1}), then we will have shown $D(\delta)\lsm_{q, \vep} \delta^{-\vep}$
which by the interpolation from Proposition \ref{interpolate} proves Theorem \ref{main}.
Proposition \ref{keystep1} allows us to climb in the first index in $M_{a, b}(\delta, \nu)$. To move this increase to the second index, we need the following.
\begin{prop}\label{climb}
If $\delta \leq \min(\nu^a, \nu^b)$, then
$$M_{a, b}(\delta, \nu) \leq M_{b, a}(\delta, \nu)^{1/2}D(\frac{\delta}{\nu^b})^{1/2}.$$
\end{prop}
\begin{proof}
This follows from $\int f^{2}g^{4} \leq (\int f^{4}g^{2})^{1/2}(\int g^{6})^{1/2}$ and applying the definition of $M_{b, a}(\delta, \nu)$ to the first term
on the right hand side and parabolic rescaling to the second term on the right hand side.
\end{proof}

\begin{lemma}
If $N \in \N$ and $\delta, \nu \in q^{-\N}$ are such that $\delta \leq \nu^{2^N}$, then
\begin{align*}
M_{1, 1}(\delta, \nu) \leq C_{1}^{2}\nu^{-O(1)}D(\frac{\delta}{\nu^{2^{N - 1}}})^{\frac{1}{3 \cdot 2^N}}D(\frac{\delta}{\nu^{2^N}})^{\frac{2}{3 \cdot 2^N}}\prod_{j = 0}^{N - 1}D(\frac{\delta}{\nu^{2^j}})^{\frac{1}{2^{j + 1}}}
\end{align*}
where $C_{1}$ is the implied absolute constant from Proposition \ref{keystep1}.
\end{lemma}
\begin{proof}
Repeatedly applying Proposition \ref{keystep1} followed by Proposition \ref{climb}
a total of $N$ times gives that
\begin{align*}
M_{1, 1}(\delta, \nu) \leq (C_{1}\nu^{-O(1)})^{2}M_{2^{N - 1}, 2^{N}}(\delta, \nu)^{1/2^N}\prod_{j = 0}^{N-1}D(\frac{\delta}{\nu^{2^j}})^{1/2^{j + 1}}.
\end{align*}
Applying Lemma \ref{mabtriv}, then completes the proof.
\end{proof}

Combining this with Proposition \ref{bired} gives the following.
\begin{lemma}\label{iter2}
If $N \in \N$ and $\delta, \nu \in q^{-\N}$ are such that $\delta \leq \nu^{2^N}$, then
\begin{align*}
D(\delta) \leq C_{2}D(\frac{\delta}{\nu}) + C_{2}C_{1}^{2}\nu^{-O(1)}D(\frac{\delta}{\nu^{2^{N - 1}}})^{\frac{1}{3 \cdot 2^N}}D(\frac{\delta}{\nu^{2^N}})^{\frac{2}{3 \cdot 2^N}}\prod_{j = 0}^{N - 1}D(\frac{\delta}{\nu^{2^j}})^{\frac{1}{2^{j + 1}}}
\end{align*}
where $C_{1}$ is the implied absolute constant from Proposition \ref{keystep1} and $C_{2}$ is the implied absolute
constant from Proposition \ref{bired}.
\end{lemma}

We now just need to iterate Lemma \ref{iter2}. To avoid having to deal with the constraints
that $\delta, \nu \in q^{-\N}$, it will be more convenient to bound
\begin{align}\label{frakdec}
\mf{D}(\delta) := \sup_{\delta_0 \in q^{-\N} \cap [\delta, 1]}D(\delta_0)
\end{align}
instead of $D(\delta)$ as $\mf{D}(\delta)$ is defined for all real $\delta \in (0, 1]$ (rather than just for $\delta \in q^{-\N}$)
and is monotonic, that is, $\mf{D}(\delta_L) \leq \mf{D}(\delta_S)$ if $\delta_L \geq \delta_S$.
(A warning: The notation in \eqref{frakdec} is reverse from what it is in \cite[Section 5]{CHLMRY}.)
Trivially we also have that if $\delta \in q^{-\N}$, then $D(\delta) \leq \mf{D}(\delta)$.
Thus to show that $D(\delta) \lsm_{q, \vep} \delta^{-\vep}$ for all $\delta \in q^{-\N}$, it suffices
to show that $\mf{D}(\delta) \lsm_{q, \vep} \delta^{-\vep}$ for all $\delta \in (0, 1]$.

\begin{prop}\label{iter3}
If $N \in \N$ and $\delta \in (0, 1]$, then
\begin{align*}
\mf{D}(\delta) \leq (q^{2^{N-1}} + C_2) &\mf{D}(q^{-1}\delta^{1 - \frac{1}{2^{N}}})\\
& + (q^{2^{N-1}} + C_{2}C_{1}^{2}q^{1/3})\delta^{-\frac{C_3}{2^N}}\mf{D}(q^{-2^{N - 1}}\delta^{1/2})^{\frac{1}{3 \cdot 2^N}}\prod_{j = 0}^{N - 1}\mf{D}(q^{-2^j}\delta^{1 - \frac{1}{2^{N - j}}})^{\frac{1}{2^{j + 1}}}
\end{align*}
where $C_{1}$ is the implied absolute constant from Proposition \ref{keystep1}, $C_{2}$ is the implied absolute constant from Proposition \ref{bired},
and $C_3$ is the implied absolute constant from the $\nu^{-O(1)}$ in Lemma \ref{iter2}.
\end{prop}
\begin{proof}
Fix $N \in \N$ and $\delta_0 \in q^{-\N} \cap [\delta, 1]$.
Suppose first that $\delta_0 \leq q^{-2^N}$. 
Let $\nu := q^{\lceil \log_{q}\delta_{0}^{1/2^N}\rceil} \geq \delta_0^{1/2^N} \geq \delta^{1/2^N}$.
Since $\delta_0 \leq q^{-2^N}$, $\nu \in q^{-\N}$.
Note that for $j = 0, 1, 2, \ldots, N$,
\begin{align*}
\frac{\delta_0}{\nu^{2^j}} = q^{\log_{q}\delta_0 - \lceil \log_{q}\delta_{0}^{1/2^N}\rceil 2^{j}} \geq \delta_{0}^{1 - \frac{1}{2^{N - j}}}q^{-2^j} \geq \delta^{1 - \frac{1}{2^{N - j}}}q^{-2^j}
\end{align*}
and hence $$D(\frac{\delta_0}{\nu^{2^j}}) \leq \mf{D}(q^{-2^j}\delta^{1 - \frac{1}{2^{N - j}}}).$$
Thus applying Lemma \ref{iter2} gives that
\begin{align*}
D(\delta_0) &\leq C_{2}D(\frac{\delta_0}{\nu}) + C_{2}C_{1}^{2}\nu^{-C_3}D(\frac{\delta_0}{\nu^{2^{N - 1}}})^{\frac{1}{3 \cdot 2^N}}D(\frac{\delta_0}{\nu^{2^N}})^{\frac{2}{3 \cdot 2^N}}\prod_{j = 0}^{N - 1}D(\frac{\delta_0}{\nu^{2^j}})^{\frac{1}{2^{j + 1}}}\\
&\leq C_2 \mf{D}(q^{-1}\delta^{1 - \frac{1}{2^{N}}}) + C_{2}C_{1}^{2}\delta^{-\frac{C_3}{2^N}}\mf{D}(q^{-2^{N - 1}}\delta^{1/2})^{\frac{1}{3 \cdot 2^N}}\mf{D}(q^{-2^N})^{\frac{2}{3 \cdot 2^N}}\prod_{j = 0}^{N - 1}\mf{D}(q^{-2^j}\delta^{1 - \frac{1}{2^{N - j}}})^{\frac{1}{2^{j + 1}}}.
\end{align*}
Applying the trivial bound that $\mf{D}(q^{-2^N}) \leq q^{2^{N}/2}$ then shows that
\begin{align*}
D(\delta_0) \leq C_2 \mf{D}(q^{-1}\delta^{1 - \frac{1}{2^{N}}}) + C_{2}C_{1}^{2}q^{1/3}\delta^{-\frac{C_3}{2^N}}\mf{D}(q^{-2^{N - 1}}\delta^{1/2})^{\frac{1}{3 \cdot 2^N}}\prod_{j = 0}^{N - 1}\mf{D}(q^{-2^j}\delta^{1 - \frac{1}{2^{N - j}}})^{\frac{1}{2^{j + 1}}}
\end{align*}
as long as $\delta_0 \leq q^{-2^N}$.
On the other hand if $\delta_0 > q^{-2^N}$, we apply the trivial bound and obtain that 
$D(\delta_0) \leq \delta_{0}^{-1/2} \leq q^{2^N/2}$.
Since $\mf{D}(\delta) \geq 1$ by Proposition \ref{lowerbound},
combining the bounds obtained in the two regimes above and then taking the supremum over all $\delta_0 \in q^{-\N} \cap [\delta, 1]$,
the completes the proof of the lemma.
\end{proof}

We now are ready to show that $\mf{D}(\delta) \lsm_{q, \vep} \delta^{-\vep}$.
Let $\ld \geq 0$ be the smallest constant such that
$\mf{D}(\delta) \lsm_{q, \vep} \delta^{-\ld - \vep}$
for all $\delta \in (0, 1]$. From the trivial bound on $D(\delta)$,
$\ld \leq 1/2$. If $\ld = 0$, then we are done and so we assume that $\ld > 0$.
Choose $N$ such that
$$\frac{5}{6} + \frac{N}{2} - \frac{C_3}{\ld} \geq 1.$$
Then Proposition \ref{iter3} implies that for all $\delta \in (0, 1]$, we have
\begin{align*}
\mf{D}(\delta) &\lsm_{q, N, \vep} \delta^{-\ld(1 - \frac{1}{2^N}) - \vep} + \delta^{-\frac{C_3}{2^N} - \frac{\ld}{6 \cdot 2^N} - \sum_{j = 0}^{N - 1}(1 - \frac{1}{2^{N - j}})\frac{1}{2^{j + 1}}\ld - \vep}\\
&\leq \delta^{-\ld(1 - \frac{1}{2^N}) - \vep} + \delta^{-\ld(1 - \frac{1}{2^N}(\frac{5}{6} + \frac{N}{2} - \frac{C_3}{\ld})) - \vep} \leq \delta^{-\ld(1 - \frac{1}{2^N}) - \vep}
\end{align*}
where the last inequality is by our choice of $N$.
But this contradicts the minimality of $\ld$. Therefore we must have $\ld = 0$
which finishes the proof.

\section{Five proofs of the key step}\label{fiveproofs}
In this section we give five different proofs of Lemma \ref{keystep3} which in turn implies Proposition \ref{keystep1}.
Before we begin our proofs, we first discuss how the efficient congruencing proof in the quadratic Vinogradov case found in
\cite[Section 4.3]{Pierce} motivated the definition of $M_{a, b}(\delta, \nu)$ and Proposition \ref{keystep1}.

\subsection{Motivation}\label{motivation}
The inspiration for the definition of $M_{a, b}(\delta, \nu)$ and Proposition \ref{keystep1}
came from \cite[Section 4.3]{Pierce} which is the quadratic Vinogradov version of Heath-Brown's simplification \cite{HB-Cubic}
of Wooley's proof \cite{W16} of cubic Vinogradov. See \cite{Pierce} and \cite{WooleyICM} for more details
on Vinogradov's Mean Value Theorem. 

The main bilinear object considered in \cite[Section 4.3]{Pierce} is
\begin{align*}
I_{1}(X; a, b) := \max_{\xi \not\equiv \eta \Mod{p}}\int_{[0, 1]^2}|\sum_{\st{n \leq X\\n \equiv \xi \Mod{p^a}}}e(\alpha n + \beta n^2)|^{2}|\sum_{\st{n \leq X\\n \equiv \eta \Mod{p^b}}}e(\alpha n + \beta n^2)|^{4}\, d\alpha\, d\beta.
\end{align*}
For fixed residue classes $\xi \Mod{p^a}$ and $\eta \Mod{p^b}$ with $\xi \not\equiv \eta \Mod{p}$,
the integral counts the number of solutions to 
\begin{align}\label{quadvmvt}
\begin{aligned}
x_{1} + x_{2} + x_{3} &= x_{4} + x_{5} + x_{6}\\
x_{1}^{2} + x_{2}^{2} + x_{3}^{2} &= x_{4}^{2} + x_{5}^{2} + x_{6}^{2}
\end{aligned}
\end{align}
with $1 \leq x_i \leq X$, $x_{1} \equiv x_{4} \equiv \xi \Mod{p^a}$, the other
$x_{i} \equiv \eta \Mod{p^b}$, and $\xi\not\equiv \eta \Mod{p}$.
The key estimate in \cite[Section 4.3]{Pierce} with regards to $I_{1}(X; a, b)$ is 
that if $1 \leq a \leq 2b$ and $p^{2b} \leq X$, then
\begin{align}\label{iab}
I_{1}(X; a, b) \leq p^{2b - a}I_{1}(X; 2b, b).
\end{align}
This should be immediately compared to Proposition \ref{keystep1}.
This loss of $p^{2b - a}$ is sharp. To see this, heuristically 
we expect $I_{1}(X; a, b) \approx X^{3}/p^{a + 2b}$ (here we use $\approx$ somewhat informally). 
Thus if we define $I_{1}'(X; a, b) := (p^{a + 2b}/X^3)I_{1}(X; a, b)$, 
\eqref{iab} now reads $I'_{1}(X; a, b) \leq I'_{1}(X; 2b, b)$
and we expect both sides to be $\approx 1$. Since we will want to prove $D(\delta) \approx 1$ (and hence $M_{a, b}(\delta, \nu) \approx 1$), we should expect $I'_{1}(X; a, b)$ to be similar to $M_{a, b}(\delta, \nu)$.

We now turn to the proof of \eqref{iab} from \cite[Lemma 4.4]{Pierce} as this will motivate our proofs of Lemma \ref{keystep3}.
Since solutions to \eqref{quadvmvt} are translation invariant, by subtracting $\eta$, we may assume that $\eta = 0$
and so $\xi\not\equiv 0 \Mod{p}$.
Write $x_{i} = \xi + p^{a}y_i$ for $i = 1, 4$ and $x_{i} = p^{b}y_i$
for $i = 2, 3, 5, 6$.
Since $x_{1}^{2}+ x_{2}^{2} + x_{3}^{2} = x_{4}^{2} + x_{5}^{2} + x_{6}^{2}$ we must have
\begin{align}\label{twosides}
(\xi + p^{a}y_1)^{2} \equiv (\xi + p^{a}y_{4})^{2} \Mod{p^{2b}}.
\end{align}
Then
$$p^{a}(y_1 - y_4)(2\xi + p^{a}(y_1 + y_4)) \equiv 0 \Mod{p^{2b}}.$$
Since $\xi \not\equiv 0 \Mod{p}$, we have $2\xi + p^{a}(y_1 + y_2) \not\equiv 0\Mod{p}$
and hence we obtain an extra congruence $y_1 \equiv y_4 \Mod{p^{2b- a}}$. This will allow us to count the solutions
efficiently. Fix $y_4 \Mod{p^{2b -a}}$. There are $p^{2b - a}$ many choices and this gives the $p^{2b - a}$
on the right hand side of \eqref{iab}. This choice then fixes $y_1$ mod $p^{2b - a}$ and hence
$x_1 \equiv x_4 \equiv \xi' \Mod{p^{2b}}$ for some $\xi'$.
Now we apply the definition of $I_{1}(X; 2b, b)$.

Broadly speaking the steps of this proof are
\begin{enumerate}
\item[1.] Use translation invariance to shift 4 of the variables to be $0 \pmod{p^b}$.
\item[2.] Since $x_{2}, x_3, x_5, x_6 \equiv 0 \pmod{p^b}$, we look at the quadratic
equation mod $p^{2b}$.
\item[3.]\label{step3} Find diagonal behavior in the remaining $x_1$ and $x_4$ variables at scale mod $p^{2b}$.
\end{enumerate}
The analogue of translation invariance in decoupling is parabolic rescaling and so we have implemented
Step 1 already in Lemma \ref{keystep3}. Another way to view Lemma \ref{keystep3} is that we decouple
the two copies of $f_{I}$ into $f_{J}$ which are at scale $\nu^{2b}$ all the while
keeping the four copies of $f_{I(0, \nu^b)}$ on both sides. This corresponds to upgrading $x_1$ and $x_4$ from mod $p^{a}$
knowledge 
to mod $p^{2b}$ knowledge while gaining no new knowledge in the other remaining variables.

We now give our five proofs of Lemma \ref{keystep3}.
On a first read and to more easily compare
the argument with the efficient congruencing argument in Section
\ref{motivation}, it may be useful to set $\nu = 1/q$.

\subsection{A first proof of Lemma \ref{keystep3}}\label{expandeverything}
Our first proof will take advantage of the fact that both exponents on the left hand side of \eqref{ks3} 
are even and resembles a C\'ordoba-Fefferman argument
(see for example \cite[Section 3.2]{Demeter-book}). This proof is a modification
of an argument that first appeared in \cite[Lemma 2.7]{Li18}.
To prove Lemma \ref{keystep3}, we prove Lemma \ref{keystep4} instead.

By definition, the left hand side of \eqref{ks4}
is equal to 
\begin{align}\label{expand}
\int_{B(0, \nu^{-2b})}\sum_{J, J' \in P_{\nu^{2b}}(I)}f_{J}\ov{f_{J'}}|f_{I(0, \nu^b)}|^{4}.
\end{align}
As $(L \times \F) \cap \bigcup_{K \in P_{\delta}}\ta_{K} = \{(\xi, \eta) \in \F^2: \xi \in L, |\eta - \xi^2| \leq \delta^2\}$, the definition of $f_{L}$ then gives
$$f_{L}(x) = \int_{|t| \leq \delta^{2}}\int_{L}\wh{f}(s, s^{2} + t)\chi(sx_1 + s^{2}x_2)\chi(tx_2)\, ds\, dt$$
for any interval $L \subset \O$.
Expanding then gives that \eqref{expand} is
\begin{align*}
\sum_{J, J' \in P_{\nu^{2b}}(I)}\int_{|t_1| \leq \delta^{2}, \ldots, |t_{6}| \leq \delta^{2}}\int_{J \times I(0, \nu^b)^{2} \times J' \times I(0, \nu^b)^{2}}(\cdots)_{1}\int_{B(0, \nu^{-2b})}\chi((\cdots)_{2})\, dx\, ds_{1} \cdots ds_{6}\, dt_{1} \cdots dt_{6}
\end{align*}
where $s_{1} \in J$, $s_{4} \in J'$, $s_{2}, s_{3}, s_{5}, s_{6} \in I(0, \nu^b)$,
\begin{align*}
(\cdots)_{1} = \prod_{i = 1}^{3}\wh{f}(s_{i}, s_{i}^{2} + t_{i})\ov{\wh{f}(s_{i + 3}, s_{i + 3}^{2} + t_{i + 3})},
\end{align*}
and
\begin{align*}
(\cdots)_{2} = \sum_{i = 1}^{3}(s_{i} - s_{i + 3})x_1 + (s_{i}^{2} - s_{i + 3}^{2})x_2 + \sum_{i = 1}^{3}(t_{i} - t_{i + 3})x_2.
\end{align*}
Since
$|t_{i}| \leq \delta^{2}$, $|x_2| \leq \nu^{-2b}$, and $\delta \leq \nu^{2b}$,
$|\sum_{i = 1}^{3}(t_{i} - t_{i + 3})x_2| \leq 1$
and so $\chi(\sum_{i = 1}^{3}(t_{i} - t_{i + 3})x_2) = 1$.
Therefore
\begin{align*}
\begin{aligned}
\int_{B(0, \nu^{-2b})}\chi((\cdots)_{2})\, dx &= \int_{B(0, \nu^{-2b})}\chi(\sum_{i = 1}^{3}(s_{i} - s_{i + 3})x_1 + (s_{i}^{2} - s_{i + 3}^{2})x_2)\, dx\\
&= \nu^{-4b}1_{|\sum_{i = 1}^{3}s_{i} - s_{i + 3}| \leq \nu^{2b}, |\sum_{i = 1}^{3}s_{i}^{2} - s_{i + 3}^{2}| \leq \nu^{2b}}
\end{aligned}
\end{align*}
by Proposition \ref{ballft}.
Since $\delta \leq \nu^{2b}$ and $s_{2}, s_{3}, s_{5}, s_{6} \in I(0, \nu^b)$, 
if we have $|s_{1}^{2} + s_{2}^{2} + s_{3}^{2} - s_{4}^{2} - s_{5}^{2} - s_{6}^2| \leq \nu^{2b}$,
then
\begin{align}\label{s1s4}
|s_{1} + s_{4}||s_{1} - s_{4}| \leq \nu^{2b}.
\end{align}
Notice that in this step, if we had set $\nu = 1/q$, then this is exactly \eqref{twosides}.
Recall that $s_{1} \in J \subset I$, $s_{4} \in J' \subset I$, and let $d(I, 0) \geq q\nu$. For any $r \in I$, write $s_{1} = r + s_{1}'$ 
and $s_{4} = r + s_{4}'$ for some $|s_{1}'|, |s_{4}'| \leq \nu^a$. By the ultrametic triangle inequality, $|s_{1} + s_{4}| = |2r + (s_{1}' + s_{4}')| \geq q\nu$
since $|2r| \geq q\nu > \nu^a \geq |s_{1}' + s_{4}'|$.
Inserting this into \eqref{s1s4}
shows that we have $|s_{1} - s_{4}| \leq (q\nu)^{-1}\nu^{2b}$ (corresponding to Step 3 in the efficient congruencing argument on Page \pageref{step3}) and hence we can insert the requirement that
$d(J, J') \leq (q\nu)^{-1}\nu^{2b}$ into $\sum_{J, J'}$ in \eqref{expand}. Applying two instances of the Cauchy-Schwarz inequality
gives that
\begin{align}\label{csarg}
\int_{B(0, \nu^{-2b})}\sum_{\st{J, J' \in P_{\nu^{2b}}(I)\\d(J, J') \leq (q\nu)^{-1}\nu^{2b}}}f_{J}\ov{f_{J'}}|f_{I(0, \nu^b)}|^{4} \lsm \nu^{-O(1)}\sum_{J \in P_{\nu^{2b}}(I)}\int_{B(0, \nu^{-2b})}|f_{J}|^{2}|f_{I(0, \nu^b)}|^{4}
\end{align}
which completes the proof.

\subsection{A second proof of Lemma \ref{keystep3}}\label{identity}
Another way to interpret solution counting and the argument in Section \ref{motivation}
is to make use of the identity $\int f(x)\, dx = \wh{f}(0)$ and ask under what conditions does
0 lie in the Fourier support of $f$.
This was the point of view taken in \cite{CHLMRY}.
Like the proof in Section \ref{expandeverything}, this second proof also makes
use of the fact that $2$ and $4$ are even, however it takes more into account
some geometric properties.

Like in \eqref{expand},
the left hand side of \eqref{ks3} is equal to
\begin{align}\label{pf2eq1}
\sum_{J, J' \in P_{\nu^{2b}}(I)}\int_{\F^2}f_{J}\ov{f_{J'}}|f_{I(0, \nu^b)}|^{4} = \sum_{J, J' \in P_{\nu^{2b}}(I)}(\wh{f_{J}} \ast \wh{\ov{f_{J'}}} \ast \wh{f_{I(0, \nu^b)}} \ast \wh{f_{I(0, \nu^b)}} \ast \wh{\ov{f_{I(0, \nu^b)}}} \ast \wh{\ov{f_{I(0, \nu^b)}}})(0).
\end{align}
Observe that
$\wh{f_{I(0, \nu^b)}}$ and $\wh{\ov{f_{I(0, \nu^b)}}}$ are supported
in $\bigcup_{K \in P_{\delta}(I(0, \nu^b))}\ta_K \subset \{(\xi, \eta): |\xi| \leq \nu^{b}, |\eta| \leq \nu^{2b}\}$. Partition
this rectangle into squares $\{\Box\}$ of side length $\nu^{2b}$. Therefore
\eqref{pf2eq1} is equal to
\begin{align}\label{pf2eq2}
\sum_{\Box}\sum_{J, J' \in P_{\nu^{2b}}(I)}(\wh{f_{J}} \ast \wh{\ov{f_{J'}}} \ast (\wh{f_{I(0, \nu^b)}} \ast \wh{f_{I(0, \nu^b)}} \ast \wh{\ov{f_{I(0, \nu^b)}}} \ast \wh{\ov{f_{I(0, \nu^b)}}})1_{\Box})(0).
\end{align}
Given a $\Box$, 
we now need to decide for which $J, J' \in P_{\nu^{2b}}(I)$ is $0$ contained
in the support of $\wh{f_{J}} \ast \wh{\ov{f_{J'}}} \ast (\wh{f_{I(0, \nu^b)}} \ast \wh{f_{I(0, \nu^b)}} \ast \wh{\ov{f_{I(0, \nu^b)}}} \ast \wh{\ov{f_{I(0, \nu^b)}}})1_{\Box}$.
Since $f_{J}$ is Fourier supported in $\bigcup_{K \in P_{\delta}(J)}\ta_K \subset \ta_J$ and similarly for $f_{J'}$, 
it suffices to ask for which $J, J' \in P_{\nu^{2b}}(I)$ is $0 \in \ta_{J} - \ta_{J'} + \Box$.
In fact, it suffices to ask something weaker. Recall the definition of $\tau_{J}$ in Section \ref{nghbd}. 
As $\ta_{J} \subset \tau_{J}$, 
we will ask for which $J, J' \in P_{\nu^{2b}}(I)$ is $0 \in \tau_{J} - \tau_{J'} + \Box$.
Since $|J| = |J'| = \nu^{2b}$, $\tau_{J}$ and $\tau_{J'}$ are squares
of side length $\nu^{2b}$. Then $\tau_{J} - \tau_{J'} + \Box$ is a square of side length $\nu^{2b}$
which contains 0. Therefore our question reduces to asking for which
$J, J' \in P_{\nu^{2b}}(I)$ is such that
\begin{align}\label{jjnu2b}
\tau_{J} - \tau_{J'} + \Box = B(0, \nu^{2b}).
\end{align}
If \eqref{jjnu2b} holds, then there exists an $s_{J} \in J$,
$s_{J'} \in J'$ and $(b_{1}, b_{2}) \in \Box$ such that
\begin{align*}
|s_{J} - s_{J'} + b_{1}| &\leq \nu^{2b}\\
|s_{J}^{2} - s_{J'}^{2} + b_{2}| &\leq \nu^{2b}.
\end{align*}
We will only be interested in the second inequality.
Since $\Box \subset \{(\xi, \eta): |\xi| \leq \nu^b, |\eta| \leq \nu^{2b}$,
$|b_{2}| \leq \nu^{2b}$ and hence $|s_{J}^{2} - s_{J'}^{2}| \leq \nu^{2b}$.
Using that $s_{J} \in J \subset I$, $s_{J'} \subset J' \subset I$, $d(I, 0) \geq q\nu$
and the same reasoning as at the end of the proof of the previous section,
we see $|s_{J} + s_{J'}| \geq q\nu$ and hence
$|s_{J} - s_{J'}| \leq \nu^{2b}(q\nu)^{-1}$.
But then by the ultrametric triangle inequality, $|s - s'| \leq \nu^{2b}(q\nu)^{-1}$
for all $s \in J$ and $s' \in J'$.

Therefore we have shown that given a $\Box$, in order for $0$ to lie
in the support of $\wh{f_{J}} \ast \wh{\ov{f_{J'}}} \ast (\wh{f_{I(0, \nu^b)}} \ast \wh{f_{I(0, \nu^b)}} \ast \wh{\ov{f_{I(0, \nu^b)}}} \ast \wh{\ov{f_{I(0, \nu^b)}}})1_{\Box}$,
we cannot have that $d(J, J') > \nu^{2b}(q\nu)^{-1}$. Thus \eqref{pf2eq2} is equal to
\begin{align*}
\sum_{\Box}\sum_{\st{J, J' \in P_{\nu^{2b}}(I)\\d(J, J') \leq \nu^{2b}(q\nu)^{-1}}}\wh{f_{J}} \ast \wh{\ov{f_{J'}}} \ast (\wh{f_{I(0, \nu^b)}} \ast \wh{f_{I(0, \nu^b)}} \ast \wh{\ov{f_{I(0, \nu^b)}}} \ast \wh{\ov{f_{I(0, \nu^b)}}})1_{\Box}.
\end{align*}
Again using that $\wh{f_{I(0, \nu^b)}} \ast \wh{f_{I(0, \nu^b)}} \ast \wh{\ov{f_{I(0, \nu^b)}}} \ast \wh{\ov{f_{I(0, \nu^b)}}}$ is supported on the $\{\Box\}$ and then finally applying $\int f(x)\, dx = \wh{f}(0)$, shows that \eqref{pf2eq1}
is equal to 
\begin{align*}
\int_{\F^2}\sum_{\st{J, J' \in P_{\nu^{2b}}(I)\\d(J, J') \leq \nu^{2b}(q\nu)^{-1}}}f_{J}\ov{f_{J'}}|f_{I(0, \nu^b)}|^{4}
\end{align*}
which by the same application of the Cauchy-Schwarz inequality as in \eqref{csarg} completes the proof.

\subsection{A third proof of Lemma \ref{keystep3}}\label{taopf}
The third proof of Lemma \ref{keystep3} relies importantly on that one of the exponents is a 2 and so suggests
us to apply Plancherel. This argument first appeared in \cite[Proposition 19]{Tao247B} and is similar to the argument in
\cite{GLYZK} which we will present in the next section. An argument similar to this one was used to prove decoupling
for fractal sets on the parabola in \cite{CDGJLM}.
By Plancherel, the left hand side of \eqref{ks3} is
\begin{align}\label{pf3eq1}
\int_{\F^2}|\sum_{J \in P_{\nu^{2b}}(I)}f_{J}f_{I(0, \nu^b)}^{2}|^{2} = \int_{\F^2}|\sum_{J \in P_{\nu^{2b}}(I)}\wh{f_{J}} \ast \wh{f_{I(0, \nu^b)}} \ast \wh{f_{I(0,\nu^b)}}|^{2}.
\end{align}
If we could show that the $\{\wh{f_{J}} \ast \wh{f_{I(0, \nu^b)}} \ast \wh{f_{I(0,\nu^b)}}\}_{J \in P_{\nu^{2b}}(I)}$ are almost pairwise orthogonal,
then combining this with another application of Plancherel in \eqref{pf3eq1} would finish the proof.

Inspired by that we only looked at the quadratic equation in Section \ref{motivation}, we
only look at the vertical projection of the supports of $\wh{f_J} \ast \wh{f_{I(0, \nu^b)}} \ast \wh{f_{I(0, \nu^b)}}$.
Since $\wh{f_{I(0, \nu^b)}}$ is supported in the strip $\{(\xi, \eta) \in \F^2: |\eta| \leq \nu^{2b}\}$, 
$\wh{f_{I(0, \nu^b)}} \ast \wh{f_{I(0, \nu^b)}}$ is also supported in the strip $\{(\xi, \eta) \in \F^2: |\eta| \leq \nu^{2b}\}$.

Next, $\wh{f_{J}}$ is supported in $\ta_J \subset \tau_J \subset S_{J} := \{(\xi, \eta) \in \F^2: |\eta - s_{J}^2| \leq \nu^{2b}\}$ for any $s_J \in J$.
The ultrametric triangle inequality then shows that
$\wh{f_{J}} \ast \wh{f_{I(0, \nu^b)}} \ast \wh{f_{I(0, \nu^b)}}$ is supported in the strip
$S_{J}$. Now we claim that the $\{S_{J}\}$ are almost pairwise disjoint.
Suppose there was a $J$ and $J'$ such that $S_{J} \cap S_{J'} \neq \emptyset$. Then the ultrametric triangle
inequality gives that $|s_{J'}^{2} - s_{J}^{2}| \leq \nu^{2b}$. But just like as in the argument in Sections \ref{expandeverything}
and \ref{identity}, since $J, J' \subset I$ and $d(I, 0) \geq q\nu$, it follows that $|s_{J} + s_{J'}| \geq q\nu$ and hence
$|s_{J} - s_{J'}| \leq \nu^{2b}(q\nu)^{-1}$ and $d(J, J') \leq \nu^{2b}(q\nu)^{-1}$. Therefore 
\eqref{pf3eq1} is equal to
\begin{align*}
\int_{\F^2}\sum_{\st{J, J' \in P_{\nu^{2b}}(I)\\d(J, J') \leq \nu^{2b}(q\nu)^{-1}}}\wh{f_{J}} \ast \wh{f_{I(0, \nu^b)}} \ast \wh{f_{I(0,\nu^b)}}\ov{\wh{f_{J'}} \ast \wh{f_{I(0, \nu^b)}} \ast \wh{f_{I(0,\nu^b)}}}
\end{align*}
which by the same application of the Cauchy-Schwarz inequality as in \eqref{csarg} gives that this is
$$\lsm \nu^{-O(1)}\int_{\F^2}\sum_{J \in P_{\nu^{2b}}(I)}|\wh{f_{J}} \ast \wh{f_{I(0, \nu^b)}} \ast \wh{f_{I(0,\nu^b)}}|^{2}.$$
Applying Plancherel then finishes the proof.

\subsection{A fourth proof of Lemma \ref{keystep3}}
In this fourth proof, we will rely on that the exponent 2 suggests the use of Plancherel, much like in the third proof. However
to avoid using that 4 is even, we will instead rely on the locally constant property. This proof is essentially the argument in \cite{GLYZK}
in the case of the parabola.

To prove Lemma \ref{keystep3}, we prove \eqref{ks4} instead.
Partition $B(0, \nu^{-2b})$ into translates of $T_{\nu^b}$, that is, vertical $\nu^{-b} \times \nu^{-2b}$ rectangles 
$\{\Box\}$. Such a partition exists by Lemma \ref{tiling} (in fact the proof of Lemma \ref{tiling} gives an explicit construction).
Therefore it suffices to show that for each $\Box$, we have
\begin{align}\label{pf4eq2}
\int_{\Box}|f_{I}|^{2}|f_{I(0, \nu^b)}|^{4} \lsm \nu^{-O(1)} \sum_{J \in P_{\nu^{2b}}(I)}\int_{\Box}|f_{J}|^{2}|f_{I(0, \nu^b)}|^{4}
\end{align}
for all $f$ Fourier supported in $\bigcup_{K \in P_{\delta}}\ta_K$. 
Since we want \eqref{pf4eq2} to be true for all $f$
Fourier supported in $\bigcup_{K \in P_{\delta}}\ta_K$, it suffices to prove
prove \eqref{pf4eq2} when $\Box = T_{\nu^b}$.
That is, we just need to show that
\begin{align}\label{pfeq3}
\int_{T_{\nu^b}}|f_{I}|^{2}|f_{I(0, \nu^b)}|^{4} \lsm \nu^{-O(1)} \sum_{J \in P_{\nu^{2b}}(I)}\int_{T_{\nu^b}}|f_{J}|^{2}|f_{I(0, \nu^b)}|^{4}
\end{align}
for all $f$ Fourier supported in $\bigcup_{K \in P_{\delta}}\ta_K$.
By the locally constant property,
$|f_{I(0, \nu^b)}|$ is constant on any translate of $T_{\nu^{b}}$.
Thus to prove \eqref{pfeq3} it suffices to just show that
\begin{align}\label{pfeq4}
\int_{|x| \leq \nu^{-b}, |y| \leq \nu^{-2b}}|f_{I}(x, y)|^{2}\, dx\, dy \lsm \nu^{-O(1)} \sum_{J \in P_{\nu^{2b}}(I)}\int_{|x| \leq \nu^{-b}, |y| \leq \nu^{-2b}}|f_{J}(x, y)|^{2}\, dx\, dy.
\end{align} 
Since in the efficient congruencing argument, we only looked at the quadratic equation,
this suggests in \eqref{pfeq4} that we freeze the $x$ variable and look only
at \eqref{pfeq4} in the $y$ variable.
Thus we need to show that for every $x_0$ such that $|x_0| \leq \nu^{-b}$, we have
\begin{align}\label{pfeq5}
\int_{|y| \leq \nu^{-2b}}|f_{I}(x_0, y)|^{2}\, dy \leq \sum_{J \in P_{\nu^{2b}}(I)}\int_{|y| \leq \nu^{-2b}}|f_{J}(x_0, y)|^{2}\, dy.
\end{align}
Rewriting the left hand side of \eqref{pfeq5} and applying Plancherel, we obtain
\begin{align}\label{pf4eq6}
\int_{\F}|\sum_{J \in P_{\nu^{2b}}(I)}f_{J}(x_0, y)1_{|y| \leq \nu^{-2b}}|^{2}\, dy = \int_{\F}|\sum_{J \in P_{\nu^{2b}}(I)}\wt{f_{J}}(x_0, \eta) \ast \nu^{-2b}1_{|\eta| \leq \nu^{2b}}|^{2}\, d\eta
\end{align}
where here we use $\wt{f}$ to just denote the Fourier transform in the second variable.
Let $F_{x_0}(\eta) := \int_{\F}\wh{f}(\xi, \eta)\chi(\xi x_0)1_{\ta_{J}}(\xi, \eta)\, d\xi$. Then $f_{J}(x_0, y) = \int_{\F}F_{x_0}(\eta)\chi(\eta y)\, d\eta$
and so $\wt{f_{J}}(x_0, y) = F_{x_0}(y)$ which is supported in the projection 
of $\ta_{J}$ onto the $y$-axis, but in Section \ref{taopf}, we saw that
such a projection is contained in $S_{J}$. 
By the ultametric triangle inequality, the convolution of $\wt{f_{J}}(x_0, \eta)$ 
with $1_{|\eta| \leq \nu^{2b}}$ does not change its support.
Since from Section \ref{taopf}, we saw that $S_{J} \cap S_{J'} = \emptyset$ 
if $d(J, J') > \nu^{2b}(q\nu)^{-1}$, by the same argument as at the end of Section \ref{taopf}, it follows that \eqref{pf4eq6}
is
\begin{align*}
\lsm \nu^{-O(1)}\int_{\F}\sum_{J \in P_{\nu^{2b}}(I)}|\wt{f_{J}}(x_0, \eta) \ast \nu^{-2b}1_{|\eta| \leq \nu^{2b}}|^{2}\, d\eta.
\end{align*}
Undoing Plancherel then gives the right hand side of \eqref{pfeq5} which
then completes the proof.

\subsection{A fifth proof of Lemma \ref{keystep3}}
The previous four proofs all rely on the fact that $\int_{\F^2} f_{J}\ov{f_{J'}}|f_{I(0, \nu^b)}|^{4} = 0$ if $d(J, J') > (q\nu)^{-1}\nu^{2b}$ or some variant.
Our final proof is qualitatively different in that it does not go through this step. The proof of Lemma \ref{keystep3} given here is in the style of Bourgain and Demeter's proof of decoupling for the paraboloid
\cite{BD15, BD17}
specialized to the case of the parabola.
The Bourgain and Demeter argument uses a different bilinearization than in \eqref{mabdef}, but the main ideas are similar in that
this proof makes use of bilinear Kakeya and wavepacket decomposition.
This argument (in the real setting) first appeared in \cite[Section 3]{Li18}.

To prove Lemma \ref{keystep3}, we first reduce to proving the $a = b$ case.
\begin{lemma}\label{keystep5}
If $1 \leq a \leq b$ and $|I| = \nu^a$, then
\begin{align}\label{ks5a}
\int_{\F^2}|f_{I}|^{2}|f_{I(0, \nu^b)}|^{4} \leq \sum_{J \in P_{\nu^b}(I)}\int_{\F^2}|f_{J}|^{2}|f_{I(0, \nu^b)}|^{4}
\end{align}
for all $f$ Fourier supported in $\bigcup_{K \in P_{\delta}}\ta_K$.
\end{lemma}
\begin{proof}
By partitioning $\F^2$ into squares of side length $\nu^{-b}$ and since
we want \eqref{ks5a} to be true for all $f$ Fourier supported in $\bigcup_{K \in P_{\delta}}\ta_K$, it suffices to show that
\begin{align*}
\int_{B(0, \nu^{-b})}|f_{I}|^{2}|f_{I(0, \nu^b)}|^{4} \leq \sum_{J \in P_{\nu^b}(I)}\int_{B(0, \nu^{-b})}|f_{J}|^{2}|f_{I(0, \nu^b)}|^{4}.
\end{align*}
Since $|f_{I(0, \nu^b)}|$ is constant on $T_{\nu^b}$ and $B(0, \nu^{-b}) \subset T_{\nu^b}$, it suffices to show that
\begin{align*}
\int_{\F^2}|\sum_{J \in P_{\nu^b}(I)}f_{J}1_{B(0, \nu^{-b})}|^{2} \leq \sum_{J \in P_{\nu^b}(I)}\int_{\F^2}|f_{J}1_{B(0, \nu^{-b})}|^{2}.
\end{align*}
But this is just an application of Plancherel since $f_{J}1_{B(0, \nu^{-b})}$
is Fourier supported on $\ta_{J} + B(0, \nu^b)$ and by Lemma \ref{supportsum},
this is equal to $\ta_J$. This completes the proof of Lemma \ref{keystep5}.
\end{proof}

Therefore to prove Lemma \ref{keystep4} (and hence Lemma \ref{keystep3}), it
suffices to assume that $a = b$.
To do this, we first show an intermediate inequality that is an application of 
wavepacket decomposition and bilinear Kakeya (and is also related to the ball inflation step in \cite{BD17}, see \cite[Lemma 3.4]{Li18}):
\begin{align}\label{pf5eq2}
\dashint_{B(0, \nu^{-2b})}&|f_{I}|^{2}|f_{I(0, \nu^b)}|^{4} \lsm \nu^{-O(1)} (\dashint_{B(0, \nu^{-2b})}|f_{I}|^{2})(\dashint_{B(0, \nu^{-2b})}|f_{I(0, \nu^b)}|^{4}).
\end{align}
where $|I| = \nu^b$, $d(I, 0) \geq q\nu$, $f$ is Fourier supported in $\bigcup_{K \in P_{\delta}}\ta_K$ and $\dashint_{B} := \frac{1}{|B|}\int_{B}$.

Partition $B(0, \nu^{-2b})$ by translates of $T_{0, I}$ and call this collection $\mc{T}(I)$.
Similarly, partition $B(0, \nu^{-2b})$ by translates of $T_{0, I(0, \nu^{b})} = T_{\nu^b}$ and call this collection $\mc{T}(I(0, \nu^b))$.
Since $f_{I}$ is Fourier supported in $\ta_I$ and $f_{I(0, \nu^b)}$ is Fourier supported in $\ta_{I(0, \nu^b)}$,
from wavepacket decomposition,  we can write
$|f_{I}|^{2} = \sum_{T \in \T(I)}c_{T}^{2}1_{T}$ and $|f_{I(0, \nu^b)}|^{4} = \sum_{T' \in \mc{T}(I(0, \nu^b))}c_{T'}^{4}1_{T'}$ 
for some nonnegative constants $c_{T}$ and $c_{T'}$.
Thus the left hand side of \eqref{pf5eq2} becomes
\begin{align}\label{pf5eq3}
\sum_{\st{T \in \mc{T}(I)\\T' \in \mc{T}(I(0, \nu^b))}}c_{T}^{2}c_{T'}^{4}\frac{|T \cap T'|}{|B(0, \nu^{-2b})|}.
\end{align}
By Lemma \ref{intersectionlemma},
\begin{align*}
\frac{|T \cap T'|}{|B(0, \nu^{-2b})|}\leq \frac{\nu^{-2b}/d(I, 0)}{\nu^{-4b}} \leq \nu^{2b}(q\nu)^{-1}.
\end{align*}
Therefore \eqref{pf5eq3} is
\begin{align*}
\leq (q\nu)^{-1}\sum_{\st{T \in \mc{T}(I)\\T' \in \mc{T}(I(0, \nu^b))}}c_{T}^{2}c_{T'}^{4}\nu^{2b}= (q\nu)^{-1}(\nu^{b}\sum_{T \in \mc{T}(I)}c_{T}^{2})(\nu^{b}\sum_{T' \in \mc{T}(I')}c_{T'}^{4})
\end{align*}
which is exactly the right hand side of \eqref{pf5eq2}.
Thus we have shown that
\begin{align*}
\int_{B(0, \nu^{-2b})}|f_{I}|^{2}|f_{I(0, \nu^b)}|^{4} \lsm \nu^{-O(1)}(\dashint_{B(0, \nu^{-2b})}|f_{I}|^{2})(\int_{B(0, \nu^{-2b})}|f_{I(0, \nu^b)}|^{4}).
\end{align*}
We next show that
\begin{align}\label{ks10}
\dashint_{B(0, \nu^{-2b})}|f_{I}|^{2} \leq \sum_{J \in P_{\nu^{2b}}(I)}\dashint_{B(0, \nu^{-2b})}|f_{J}|^{2}.
\end{align}
This would finish the proof of \eqref{ks4} since $|J| = \nu^{2b}$ and so $|f_{J}|$ is constant on $B(0, \nu^{-2b})$ by Corollary \ref{constcor} and hence 
$$(\dashint_{B(0, \nu^{-2b})}|f_{J}|^{2})(\int_{B(0, \nu^{-2b})}|f_{I(0, \nu^b)}|^{4}) = \int_{B(0, \nu^{-2b})}|f_{J}|^{2}|f_{I(0, \nu^b)}|^{4}.$$
Thus it remains to show \eqref{ks10}. By Plancherel,
\begin{align*}
\int_{B(0, \nu^{-2b})}|f_{I}|^{2} &= \int_{\F}|\sum_{J \in P_{\nu^{2b}}(I)}f_{J}1_{B(0, \nu^{-2b})}|^{2}\\
& = \int_{\F}|\sum_{J \in P_{\nu^{2b}}(I)}\wh{f_J} \ast \wh{1}_{B(0, \nu^{-2b})}|^{2} = \sum_{J \in P_{\nu^{2b}}(I)}\int_{B(0, \nu^{-2b})}|f_{J}|^{2}.
\end{align*}
where in the last equality we have used that
the support of $\wh{1}_{B(0, \nu^{-2b})}$ is $B(0, \nu^{2b})$ and since $\ta_{J} + B(0, \nu^{2b}) = \ta_{J}$, each term in the $\sum_{J}$ are orthogonal. 

\bibliographystyle{amsplain}
\bibliography{qpdec}
\end{document}